\documentclass[10pt,hyp,english]{amsart}
\usepackage{amsmath, amscd, amsthm} 
\usepackage{amssymb, amsfonts, latexsym} 

\usepackage{booktabs}
\usepackage{mathtools}
\usepackage{enumerate}

\usepackage{color}

\usepackage[margin=1.1in]{geometry}
\usepackage{palatino}
\usepackage{mathrsfs}
\usepackage[svgnames]{xcolor}

\usepackage{url}
\usepackage{breakurl}
\usepackage[breaklinks]{hyperref}

\usepackage{babel}
\usepackage{tikz}
\usepackage{tikz-cd}
\usepackage{array}

\usepackage[poorman]{cleveref}
\usepackage[backend=biber,giveninits=true,hyperref,style=alphabetic]{biblatex}
\usepackage{bbm}
\DeclareFieldFormat{postnote}{#1}
\DeclareFieldFormat{multipostnote}{#1}
\bibliography{main.bib}

\newcommand{\A}{\mathcal{A}}
\newcommand{\oA}{\overline{\A}}
\newcommand{\E}{\mathbf{E}}
\newcommand{\F}{\mathbf{F}}
\newcommand{\Z}{\mathbb{Z}}
\newcommand{\N}{\mathbb{N}}
\newcommand{\Q}{\mathbb{Q}}

\renewcommand{\S}{\mathbf{S}}

\newcommand{\OFS}{\mathcal{O}_{F,\mathcal{S}}}

\newcommand{\Gm}{\mathbb{G}_m}

\newcommand{\MSL}{\mathbf{MSL}}

\newcommand{\GW}{\mathbf{GW}}
\newcommand{\W}{\mathbf{W}}

\newcommand{\C}{\mathbb{C}}

\newcommand{\HZ}{\mathbf{H}\Z}

\newcommand{\cHZ}{\wt{\mathbf{H}}\Z}

\newcommand{\wHZ}{\mathbf{H}\widetilde{\Z}}
\newcommand{\HW}{\mathbf{H}W}

\newcommand{\oH}{\overline{H}}
\newcommand{\R}{\mathbb{R}}
\newcommand{\SH}{\mathcal{SH}}
\newcommand{\Top}{{\mathcal{T}op}}

\newcommand{\Hom}{\operatorname{Hom}}
\newcommand{\RHom}{\operatorname{RHom}}
\newcommand{\Ext}{\operatorname{Ext}}
\newcommand{\Tor}{\operatorname{Tor}}

\renewcommand{\lim}{\operatorname{lim}}

\newcommand{\hoeq}{\operatorname{hoeq}}

\newcommand{\cofib}{\operatorname{cofib}}
\newcommand{\tensor}{\otimes}
\newcommand{\cotensor}{\square}
\newcommand{\dual}{\vee}

\newcommand{\Sq}{\operatorname{Sq}}
\newcommand{\wt}[1]{\widetilde{#1}}
\newcommand{\ol}[1]{\bar{#1}}
\newcommand{\ul}[1]{\underline{#1}}
\newcommand{\et}{\acute{e}t}

\newcommand{\im}{\operatorname{im}}
\newcommand{\coker}{\operatorname{coker}}
\newcommand{\pr}{\operatorname{pr}}

\newcommand{\vcd}{\operatorname{vcd}}
\newcommand{\cd}{\operatorname{cd}}
\newcommand{\f}{\mathsf{f}}

\newcommand{\eff}{\text{eff}}

\newcommand{\heart}{\heartsuit}

\newcommand{\sep}{\text{sep}}
\renewcommand{\et}{\text{\'{e}t}}

\renewcommand{\star}{{\ast,\ast}}
\newcommand{\twocomp}{{}^{{\kern -.7pt}\wedge}_2}
\newcommand{\ellcomp}{{}^{{\kern -.5pt}\wedge}_{\ell}}
\newcommand{\vsep}{\,\vert\,}

\theoremstyle{theoremstyle}
\newtheorem{theorem}{Theorem}[section]
\newtheorem*{theorem*}{Theorem}
\newtheorem{lemma}[theorem]{Lemma}
\newtheorem{proposition}[theorem]{Proposition}
\newtheorem*{proposition*}{Proposition}
\newtheorem{corollary}[theorem]{Corollary}
\newtheorem*{corollary*}{Corollary}
\newtheorem{remark}[theorem]{Remark}
\newtheorem{remark*}{Remark}

\newtheorem{defn*}{Definition}
\theoremstyle{definition}

\theoremstyle{theoremstyle}

\thanks{The authors were partially supported by the RCN Frontier Research Group Project no.~250399.}

\begin{document}
\title{Strong convergence in the motivic Adams spectral sequence}
\subjclass[2010]{14F42 (primary), 55S10, 55T15 (secondary)}
\keywords{Motivic homotopy theory, motivic Adams spectral sequence, motivic cohomology, motivic stable homotopy group}

\author{Jonas Irgens Kylling}
\address{Department of Mathematics, University of Oslo, Norway}
\email{jonasik@math.uio.no}
\author{Glen Matthew Wilson}
\address{Department of Mathematical Sciences, Norwegian University of Science and Technology, Norway}
\email{glen.m.wilson@ntnu.no}

\begin{abstract}
We prove strong convergence results for the motivic Adams spectral sequence of the sphere spectrum over fields with finite virtual cohomological dimension at the prime $2$, and over arbitrary fields at odd primes.
We show that the motivic Adams spectral sequence is not strongly convergent over number fields.
As applications we give bounds on the exponents of the $(\ell,\eta)$-completed motivic stable stems,
and calculate the zeroth $(\ell,\eta)$-completed motivic stable stems.
\end{abstract}
\maketitle

\tableofcontents

\section{Introduction}

The analogue of the Adams spectral sequence in motivic homotopy theory has been used successfully to calculate many interesting motivic invariants.
However, the convergence properties of the motivic Adams spectral sequence are more subtle than in topology, which has limited its use to fairly specialized situations.
Not even for the motivic sphere spectrum is strong convergence of the motivic Adams spectral sequence over an arbitrary base field guaranteed.

The goal of this work is to study the vanishing of the derived $E_\infty$-term of the motivic Adams spectral sequence of the sphere spectrum over a general field.
This vanishing is the crucial ingredient to obtain strong convergence from conditional convergence, following Boardman \cite{Boardman}.
For formal reasons the abutment of the motivic Adams spectral sequence is the $H$-completion of the target spectrum, where $H$ is the mod $\ell$ motivic cohomology spectrum.
The $H$-completion was identified to be the $(\ell,\eta)$-completion of the target spectrum by Hu, Kriz and Ormsby \cite{HKO:adams} under some strict bounded cellularity assumptions on the target spectrum.
Recently a different proof was given by Mantovani \cite{Mantovani} which drops the cellularity assumption.
Combined with vanishing of the derived $E_\infty$-term this implies strong convergence to the homotopy groups of the $(\ell,\eta)$-completed spectrum.

In topology it is usually sufficient for the target spectrum to satisfy some finiteness assumptions on its cohomology for the Adams spectral sequence to be strongly convergent.
Working at the prime 2, this assumption is usually sufficient in motivic homotopy theory over fields with finitely many square classes as well.
But there are many interesting fields with infinitely many square classes,
for instance the rationals, over which we would like to compute motivic invariants.

One approach could be to consider the $E_2$-page not as an $\F_2$-module,
but as a module over mod 2 Milnor $K$-theory $K^M_*(F)/2$, and then hope that the successive $E_r$-pages remain finitely generated $K^M_*(F)/2$-modules.
However, if $K^M_*(F)/2$ is not Noetherian (equivalently, there are infinitely many square classes),
then kernels of finitely generated $K^M_*(F)/2$-modules need not be finitely generated.
And indeed this issue occurs in practice when computing the $E_2$-page for the motivic sphere spectrum over the rationals with the $\rho$-Bockstein spectral sequence.
Over $\Q$ the $E_1$-page is a finitely generated $K^M_*(\Q)/2$-module.
However, on some groups the $d_1$-differential is given by multiplication by $\rho$,
hence the kernel contains a copy of the $\rho$-torsion in $K^M_*(\Q)/2$,
which is an infinitely generated $K^M_*(\Q)/2$-module.

What saves the day are vanishing regions in the $E_2$-page analogous to the vanishing lines we have in topology.
For a good motivic spectrum $X$ the $E_2$-page of the motivic Adams spectral sequence is $\Ext_{\A_\star^F}(H_\star(F;\Z/2), H_\star(X;\Z/2))$.
These $\Ext$-groups are rather computable,
and most of this paper is concerned with their properties.
For the motivic sphere spectrum over the real numbers at the prime 2, Guillou and Isaksen have shown that there are vanishing regions for $\Ext^{s,(t,w)}_{\A_\star^\R}(H_\star(\R;\Z/2),H_\star(\R;\Z/2))$ when $t - s > 0$ and $t - s - w \neq 0$.
However when $t - s = 0$ and $t - s - w \geq 0$ is even there is an infinite $h_0$-tower in the $\Ext$-group.
Over a general field this infinite $h_0$-tower is tensored with mod 2 Milnor $K$-theory, and is spread out in the cone $w \leq t - s \leq 0$ on the $E_2$-page.
Inside of this cone there can be infinitely many differentials exiting a particular tridegree of the motivic Adams spectral sequence, and it is not clear if the derived $E_\infty$-terms vanish here.
With the help of the extension
\begin{align*}
0
\to \Ext^{s}_{\A_\star^\R}(H_\star^\R, H_\star^\R(\S)) \tensor_{\Z/2[\rho]} k_*
&\to \Ext^{s}_{\A_\star}(H_\star, H_\star(\S)) \\
&\to \Tor_1^{\Z/2[\rho]}(\Ext^{s+1}_{\A_\star^\R}(H_\star^\R, H_\star^\R(\S)), k_*)
\to 0
\end{align*}
we extend the vanishing lines of Guillou and Isaksen to general fields.
This is where the assumption on the finiteness of the virtual cohomological dimension of the base field enters, since if not the outer terms in the extension can be nonzero for fixed stem and weight and any $s$.
At odd primes we use a Bockstein spectral sequence to extend the topological vanishing lines of Adams to the motivic $\Ext$-groups.

At odd primes in high Adams filtration the $E_r$-pages of the motivic Adams spectral sequences for the sphere spectrum and the motivic cohomology spectrum are the same. Thus we would expect there to be infinitely many differentials leaving a particular tridegree over special base fields (of course, this is only necessary for the derived $E_\infty$-term to be nonzero).
Hence it seems reasonable to doubt that the motivic Adams spectral sequence is strongly convergent in general.
In \Cref{cor:S-notconv} we show that the motivic Adams spectral sequence is not strongly convergent over number fields.
However, with the help of vanishing regions in the $\Ext$-groups we can show strong convergence for the motivic Adams spectral sequence for the sphere spectrum in positive stems.
For $\S/\ell^n$ the outlook is not as grim, and we prove strong convergence for $\S/\ell^n$ over any field of characteristic not $\ell$ with finite virtual cohomological dimension if $\ell = 2$ in \Cref{cor:S/elln-conv}.
This suggests a general strategy for adapting classical proofs with the Adams spectral sequence to the motivic setting. Prove the required properties for $\S/\ell^n$ and then pass to the limit over $n$, assuming that the properties are well behaved.

Motivic cohomology is often considered a generalization of singular cohomology to smooth schemes. However, other generalizations are possible; a notable one is the generalized motivic cohomology spectrum $\cHZ$ of Bachmann \cite{Bachmann:veff}.
This spectrum is a closer approximation to the motivic sphere spectrum.
In particular they have the same zeroth motivic homotopy group.
For us the likeness manifests itself in their $\Ext$-groups.
It is simpler to describe the vanishing region as an isomorphism of $\Ext$-groups in high filtration between the motivic sphere spectrum and $\cHZ$,
and then compute the $\Ext$-group of $\cHZ$.
As part of this we compute the motivic cohomology of several motivic spectra and homotopy modules related to $\cHZ$.
This seems to be an interesting computation in itself.
For instance, motivic cohomology of the Witt $K$-theory sheaf is
$H^\star(\ul{K}_{*}^W) = \Sigma^{1,-1}\A^\star/\A^\star(\tau, \Sq^2 + \rho\Sq^1)$,
which as far as we know has no classical analogue (its complex realization is zero).

As a simple corollary of our computations we give a bound on the exponent of the torsion in the completed motivic stable stems.
For instance, in positive stems the exponent of the $2$-torsion in $\pi_{t,w}(\S_{2,\eta}^\wedge)$ is bounded by
$2^{\max\{\lceil(t-w + 1)/2\rceil, t + \vcd(F)\}}$.
This is a partial answer to a question in \cite{ALP}.
As a second application we compute the $(2,\eta)$-completed zero-line of the motivic stable stems in \Cref{sec:zero-line}.
This is essentially the same computation as done for $\pi_{0,0}(\S_{2,\eta}^{\wedge})$ by Morel in \cite{Morel:adams}.
The computation makes Morel's pull-back square \cite[Theoreme 5.3]{Morel:witt} appear less mysterious, since the pull-back square is suggested by the computation of the $\Ext$-group.

We hope this work will be useful to applications of the motivic Adams spectral sequence over general base fields.
Up to now, most work with the motivic Adams spectral sequence has been at the prime 2 over $\C$, $\R$, or other fields with finitely many square classes.

\subsection*{Previous work}
One of the first applications of the motivic Adams spectral sequence was Morel's proof of Milnor's conjecture on quadratic forms in \cite{Morel:adams}.
The proof consists of a computation of $\pi_{0,0}(\S_{2,\eta}^{\wedge})$.
A more systematic study of the motivic Adams spectral sequence was carried out by Dugger and Isaksen \cite{DI:adams}.
They calculate the $\C$-motivic stable stems in a large range.
Later Dugger and Isaksen \cite{DI:Real} computed the first four Milnor-Witt stems over the real numbers by using the $\rho$-Bockstein spectral sequence to calculate $\Ext_{\A_\star^\R}(H_\star(\R;\Z/2), H_\star(\R;\Z/2))$ in a range
and observing that the motivic Adams spectral sequence collapses in the first four Milnor-Witt stems.
In a sequel they establish a comparison between the real and the $C_2$-equivariant stable stems \cite{DI:Z2-R}.

Unlike in topology, the motivic Hopf map $\eta \in \pi_{1,1}(\S)$ is not nilpotent. 
This has spurred several investigations into the properties of the eta-inverted sphere $\S[\eta^{-1}]$. 
Ananyevskiy, Levine, and Panin \cite{ALP} study the $\eta$-inverted motivic sphere spectrum. As a corollary they show $\pi_{m+n,n}(\S)$ is torsion for $m > 0,n \geq 0$. We give bounds on the exponent of the torsion in \Cref{sec:bounds}.
Guillou and Isaksen \cite{GI:eta} work out the calculation of the $h_1$-inverted motivic Adams spectral sequence for the sphere spectrum over the complex numbers, assuming a certain pattern of differentials.
Andrews and Miller \cite{AM:eta} verify that the differentials conjectured by Guillou and Isaksen hold to give a complete calculation of $\pi_{**}(\S\twocomp[\eta^{-1}])$.
Guillou and Isaksen \cite{GI:Real} perform the same calculation, but over the real numbers. Along the way they establish a vanishing region in $\Ext_{\A_\star^\R}(H_\star(\R;\Z/2), H_\star(\R;\Z/2))$.
Wilson \cite{W:EtaRational} extends this calculation to the rational numbers $\Q$ and fields with 2-cohomological dimension at most 2. 
The calculation over $\Q$ is notable since there are infinitely many square classes in $\Q$, as $K^M_1(\Q)/2 = \Q^{\times}/2$ is generated by the classes $[-1]$ and $[p]$ for $p$ a prime.
The computation is still successful since it is possible to determine all the differentials in the spectral sequence.

Calculations of $\Ext$-groups over algebraically closed fields can be done with many of the same techniques as in topology, and also by comparison with topology by complex realization.
In this case, the motivic May spectral sequence can be used to great success, as the work of Isaksen, Dugger, and Guillou shows \cite{DI:adams}, \cite{DI:Real}, \cite{GI:eta}, \cite{Isaksen:A2}.
Unfortunately, the motivic May spectral sequence is not available over fields in which $-1$ is not a square. The main tool for computing $\Ext$-groups over such fields is the $\rho$-Bockstein spectral sequence.
The $\rho$-Bockstein spectral sequence was introduced by Hill \cite{Hill} to calculate the $\Ext$-groups of truncated Brown-Peterson spectra and the very effective connective cover of hermitian $K$-theory over the real numbers.
In \cite{Wilson:web} there are machine calculations and figures of $\Ext$-groups over various base fields in a range.

Ormsby and {\O}stv{\ae}r \cite{OP:low} use the motivic Adams-Novikov spectral sequence over fields of cohomological dimension at most 2 to compute the first stable stem of the sphere spectrum.
The low cohomological dimension avoids many of the complications encountered in this paper.
In \cite{OP:BP} they use the motivic Adams spectral sequence to compute the coefficients of motivic (truncated) Brown-Peterson spectra over the rationals. 
In this case it is possible to determine all the differentials exactly.
Wilson and {\O}stv{\ae}r \cite{WP:finite} compute $\pi_\star(\S)$ in a certain region over finite fields.
In \cite{Ormsby} Ormsby computes the coefficients of motivic (truncated) Brown-Peterson spectra over $p$-adic fields.
Because of the Teichm\"{u}ller lift $p$-adic fields have finitely many square classes.
In \cite{LYZ} Levine, Yang and Zhao use the motivic Adams spectral sequence to compute the coefficients of $\MSL_{\ell,\eta}^{\wedge}$ along the $(2n,n)$-diagonal over any perfect field. 

\subsection*{Organization of this paper}
We begin in \Cref{sec:recollection} with a quick recap of motivic cohomology, the motivic (dual) Steenrod algebra, the motivic Adams spectral sequence and its $E_2$-page. 
The main focus is on structure results of the motivic Steenrod algebra and motivic cohomology needed to describe the properties of the cobar complex and the $\Ext$-groups.
Next we use these tools in \Cref{sec:detour} to compute the mod 2 motivic cohomology of $\wHZ$ and related spectra defined by \cite{Bachmann:veff}.
These calculations form the input to \Cref{sec:2} where we study the mod 2 $\Ext$-groups over fields of finite virtual cohomological dimension.
Somewhat simpler are the mod $\ell$ $\Ext$-groups which are treated in \Cref{sec:odd} with a Bockstein spectral sequence.
Combining the two previous sections allow us to prove our main convergence results in \Cref{sec:strong}. 
Here we use the vanishing regions to show that the derived $E_\infty$-terms vanish in positive stems over arbitrary fields.
In nonnegative stems the vanishing of the derived $E_\infty$-term is in general unclear, but we prove strong convergence for $\S/\ell^n$.
With the hope of resolving the convergence in negative stems we study the $\ell$-Bockstein spectral sequence for motivic cohomology and its convergence properties in \Cref{sec:Bockstein}. 
This is the same as the mod $\ell$ motivic Adams spectral sequence for the motivic cohomology spectrum $\HZ$.
We show that the $\ell$-Bockstein spectral sequence is not strongly convergent, and as a corollary we get that the motivic Adams spectral sequence for the sphere spectrum is not strongly convergent over number fields.
We end with two short sections with applications to calculations of motivic homotopy groups. 
In \Cref{sec:bounds} we give bounds on the torsion in the $(\ell,\eta)$-completed stables stems. 
In \Cref{sec:zero-line} we calculate the zeroth motivic homotopy groups of the $(\ell,\eta)$-completed sphere spectrum.

\subsection*{Acknowledgments}
The authors are grateful to Ivan Panin and John Rognes for stimulating discussions and questions which prompted the current work.
Without the work of Lorenzo Mantovani on localizations and completions in motivic homotopy theory we would probably not have ventured to write this paper, and we thank him for explaining his work to us.
The authors thank Oliver R\"{o}ndigs for pointing out \Cref{rmk:Oliver},
and Paul Arne {\O}stv{\ae}r for helpful comments.

\subsection*{Notation}
Throughout the paper $\ell$ will be a fixed prime number.
We will work in the stable motivic homotopy category $\SH(F)$ over a base field $F$ of characteristic not $\ell$.
We do not require $F$ to be perfect unless explicitly stated.
The following table summarizes the notation used in the paper:
\begin{center}
\footnotesize
\begin{tabular}{l|l}
$\ell$, $F$ & a prime number, a field of characteristic not $\ell$ \\
$\SH(F)$ & motivic stable homotopy category \\
$X$, $\S$  & motivic spectrum, motivic sphere spectrum \\
$\wHZ$, $\HZ$, $\HW$ & (generalized) motivic cohomology, Witt motivic cohomology spectra \\
$\HZ/\ell = H$, $\ol{H}$ & mod $\ell$ motivic cohomology spectrum, $\cofib(\S \to H)$ \\
$H_{-p,-q} = H^{p,q} = H^{p,q}(F;\Z/\ell)$ & mod $\ell$ motivic cohomology of $F$ \\
$S^{p,q}, \Sigma^{p,q}(-)$ & $(S_s^1)^{p-q}\wedge \Gm^{\wedge q}$, $S^{p,q}\wedge-$ \\
$\pi_{p,q}(X)$ & $[S^{p,q}, X]_{\SH(F)}$ \\
$H^{*,*}(X)$ & $[X, \Sigma^{*,*}\HZ]_{\SH(F)}$, mod $\ell$ motivic cohomology of $X$ \\
$H_{*,*}(X)$ & $\pi_{*,*}(X\wedge\HZ)$, mod $\ell$ motivic homology of $X$ \\
$K^M_* = K^M_*(F)$, $k_*$ & Milnor $K$-theory of $F$, mod $\ell$ \\
$\ul{K}_{*}^{MW}$, $\ul{K}_{*}^{M}$, $\ul{K}_{*}^{W}$, $\ul{K}_{*}^{M}/2$ & Milnor-Witt-, Milnor-, Witt- and mod 2 Milnor $K$-theory sheaves \\
$h, \eta$ & hyperbolic plane, Hopf map \\
$\A^\star, \A_\star$, $\Delta$, $\epsilon$, $\eta_L, \eta_R$ & motivic Steenrod algebra, dual, structure maps of dual \\
$\beta$, $\Sq^i$, $P^i$ & Bockstein, motivic Steenrod squares, motivic power operations \\
$\tau$, $\rho$, $\xi_i$, $\tau_i$ & $H_\star$-algebra generators of $\A_\star$ \\
$\A^\star(1)$ & $H^\star$-subalgebra of $\A^\star$ generated by $\Sq^1$ and $\Sq^2$ \\
$E_r^{s,t,w}(X) = E_r^{s,t,w}$ & $E_r$-page of the motivic Adams spectral sequence of $X$ \\
$\Ext_{\A_\star}^{s,(t,w)}(H_\star, H_\star(X))$ & left $\A_\star$-comodule $\Ext$ of $H_\star(X)$ \\
$C^\bullet(H_\star, H_\star(X))$ & cobar complex computing $\Ext_{\A_\star}(H_\star, H_\star(X))$ \\
$s, t, w$ & cohomological filtration, topological degree, motivic weight \\
$t - s$, $t - s- w$ & stable stem, Milnor-Witt degree \\
$W(F)$, $I^{n}(F)$, $I^{n}$ & Witt ring of $F$, fundamental ideal of $W(F)$, $I^n = W(F), n \leq 0$ \\
$\cd_{\ell}(F)$, $\vcd(F) = \cd_2(F(\sqrt{-1}))$ & (virtual) $\ell$-cohomological dimension of $F$ \\
$\langle\!\langle a_1, \dots, a_n \rangle\!\rangle$ & Pfister form
\end{tabular}
\end{center}

\section{Recollections on mod $\ell$ motivic cohomology, the motivic Steenrod algebra and the motivic Adams spectral sequence}
\label{sec:recollection}
In this section we fix the notation used later and recall results on the structure of mod $\ell$ motivic cohomology, the motivic Steenrod algebra and the motivic Adams spectral sequence. As usual it is necessary to treat the prime 2 distinctly from the other primes.

Recall that
$\A^\star = [\HZ/\ell, \Sigma^\star \HZ/\ell]_{\SH(F)}$ and 
$\A_\star = \pi_{\star}(\HZ/\ell\wedge\HZ/\ell)$
are the motivic Steenrod algebra and its dual,
with coefficients $H^\star(F;\Z/\ell) = H^\star = H_{-*,-*}$.
The dual is a Hopf algebroid
with structure maps $\eta_L,\eta_R : \A_\star \to H_\star, \epsilon : \A_\star \to H_\star, \Delta: \A_\star \to \A_\star \tensor_{H_\star} \A_\star$ \cite{Voevodsky:power}, \cite{HKO:steenrod}.
Here the tensor product $-\tensor_{H_\star}-$ is formed by considering $\A_\star$ as a left $H_\star$-module with $\eta_L$ and as a right $H_\star$-module with $\eta_R$.
For any prime $\ell$ we have $H_{p,q} = 0$ for $p < q$ or $p > 0$,
and an identification with mod $\ell$ Milnor $K$-theory of the base field along the diagonal $H_{n,n} \cong k_{-n}(F)$.
Note that the subgroup $k_* \subset \A_\star$ is a central subalgebra.

A motivic spectrum is $\A_\star$-good if the canonical map $(\A_\star \tensor_{H_\star} H_\star(X))_{p,q} \to [S^{p, q}, H \wedge H \wedge X]$ is an isomorphism.
In this case $H_\star(X)$ is a left $\A_\star$-comodule by the map 
\[
\Psi_X : [S^{p,q}, H\wedge X]
\to [S^{p, q}, H \wedge H \wedge X]
\cong (\A_\star \tensor_{H_\star} H_\star(X))_{p,q},
\]
where the first map is induced by the unit $\S \to H$.
All cellular spectra are $\A_\star$-good by \cite[Lemma 7.6]{DI:adams}.
\subsection*{The prime 2}
When $\ell = 2$ we have $H_\star = k_*[\tau]$ (cf.~\cite[2.1]{DI:adams}),
with $\tau \in H_{0,-1} = \mu_2(F^\times)$ represented by $-1$.
There is also a canonical element $\rho \in H_{-1,-1} = F^\times/(F^{\times})^2$ represented by $-1$.
The mod 2 motivic Steenrod algebra is generated as an $H^\star$-algebra by the motivic Steenrod squares $\Sq^i \in \A^{i,\lfloor i/2 \rfloor}$,
subject to the motivic Adem relations \cite[Theorem 4.5.1]{Riou}, \cite[Theorem 5.1]{HKO:steenrod}.
By \cite[12]{Voevodsky:power}, \cite[Theorem 5.6]{HKO:steenrod} the dual has the algebra structure
\begin{equation*}
\A_\star = H_\star[\tau_0, \tau_1, \dots, \xi_1, \xi_2, \dots]/(\tau_i^2 + (\tau + \rho\tau_0)\xi_{i+1} + \rho\tau_{i+1}).
\end{equation*}
Here $\xi_i$ is in bidegree $(2^{i+1}-2, 2^i - 1), i > 0$,
and $\tau_i$ is in bidegree $(2^{i+1}-1, 2^i - 1), i \geq 0$.
Let $H_\star^\C = H^{-(\star)}(\C;\Z/2) = \Z/2[\tau]$
and $H_\star^\R = H^{-(\star)}(\R;\Z/2) = \Z/2[\tau,\rho]$,
and similarly for the dual motivic Steenrod algebra, $\A_\star^\C$ and $\A_\star^\R$.
A consequence of the structure of $H_\star$ and $\A_\star$ is that \cite[Theorem 5.6]{HKO:steenrod}
\begin{align}
\label{eq:dual}
H_\star = H_\star^\R\tensor_{\Z/2[\rho]} k_*\qquad \text{and}\qquad
\A_\star = \A_\star^\R \tensor_{\Z/2[\rho]} k_*.
\end{align}
Here $\A_\star^\R$ is the motivic Steenrod algebra over $\R$,
and $k_*$ is considered as a $\Z/2[\rho]$-module by mapping $\rho$ to $\rho$.

\begin{remark} %
\label{rmk:Oliver}
The smallest integer $n$ (if it exists) such that $\rho^n = 0 \in k_n(F)$ is the smallest integer $n$ such that $-1$ is a sum of $2^n$ squares \cite[Corollary 3.5]{ElmanLam}, \cite[Theorem III.4.5]{MilnorHusemoller}.
For the field $\Q(x_0, \dots, x_{2^n})/(\sum_i x_i^2 = 0)$ we have $\rho^n = 0$ and $\rho^{n-1} \neq 0$ \cite[Satz 5]{Pfister}.
\end{remark}

\begin{remark}
\label{rmk:fin-fin-vcd}
Let $F$ be a field of odd characteristic with finite virtual cohomological dimension.
Then $F$ has finite $2$-cohomological dimension.
Indeed, if $F$ has finite virtual cohomological dimension then $2 I^n(F) = I^{n+1}(F)$ for $n > \vcd(F)$
\cite[Corollary 35.27]{EKM}, \cite[Korollar 1]{AP}.
But if $\F_q$ is the prime field of $F$
then $W(F)$ is a $W(\F_q)$-algebra (i.e., $\Z/2\oplus\Z/2$- or $\Z/4$-algebra), hence $I^{n+2} = 0$.
\end{remark}

\subsection*{Odd primes}
For $\ell$ an odd prime the structure of the motivic Steenrod algebra is simpler,
while the coefficients $H^\star$ are harder to describe.
\begin{lemma} %
\label{lem:Hp}
Let $F$ be a field and let $d$ be the degree of the extension
$F(\zeta_{\ell})/F$ where $\zeta_{\ell}$ is a primitive $\ell$'th root of
unity. Multiplication by the generator
$\zeta_\ell \in H^{0}_\et(F;\mu_{\ell}^{\tensor d})
\cong \mu_\ell(F^\sep)
\cong \Z/\ell$
induces an isomorphism
\[
\zeta_\ell : H^{p}_\et(F;\mu_{\ell}^{\tensor q})
\xrightarrow{\cong}
H^{p}_\et(F;\mu_{\ell}^{\tensor(q + d)}).
\]
\end{lemma}
\begin{proof}
Let $G$ be the Galois group of the extension $F^{\sep}/F$.
Then
\[
H^{p}_\et(F;\mu_{\ell}^{\tensor q})
=
\Ext_{\Z[G]}^p(\Z, \mu_{\ell}^{\tensor q}(F^{\sep})),
\]
and $\mu_{\ell}^{\tensor q}(F^{\sep}) \cong \Z/\ell$ as an abelian group,
with the action induced by the diagonal action of the Galois group on each tensor factor.
That is $g \in \Z[G]$ acts on $\mu_{\ell}^{\tensor q}(F^{\sep}) \cong \mu_{\ell}(F^\sep) \cong \Z/\ell$ as $g^q$ (see for instance \cite[p.~46]{Milne}). %
The above statements are then easy to check.
\end{proof}
As a corollary of \Cref{lem:Hp} we see that $H^\star$ is generated by elements in $H^{p,q}$ such that
$p \leq q$ and $q - \ell < p$.
Let $\wt{H}^\star$ be the graded $\Z/\ell$-submodule of such elements.
That is, any element in $H^\star$ is a multiple of $\zeta_\ell$ and a unique element of $\wt{H}^\star$.
As an $H^\star$-algebra the mod $\ell$ motivic Steenrod algebra is generated by the Bockstein $\beta$ in bidegree $(1,0)$ and the $i$'th power operation $P^i$ in bidegree $(2i(\ell -1),i(\ell -1))$.
Hence the action of $P^i$ on $\wt{H}^\star$ is trivial for degree reasons.
Thus the right unit $H_\star \to \A_\star$ is completely determined by the action of the Bockstein and the action of $P^i$ on $\zeta^l$.
By the Adem relations we have
$P^{i}(\zeta^k) = \zeta P^{i}(\zeta^{k-1}),$
hence
$P^i(\zeta^k) = 0, i > 0$.
Consequently the Bockstein is the only element of $\A^\star$ which can act nontrivially on $H^\star$ (this should be contrasted with the mod 2 case, where many elements of $\A^\star$ act nontrivially on powers of $\tau$).

The algebra structure of the dual mod $\ell$ motivic Steenrod algebra is \cite[Theorem 12.6]{Voevodsky:power}, \cite[Theorem 5.6]{HKO:steenrod}
\[
A_\star \cong H_\star[\tau_0, \tau_1, \dots, \xi_1, \xi_2, \dots]/(\tau_0^2, \tau_1^2, \dots).
\]
Here $\xi_i$ is in bidegree $(2\ell^i-2, \ell^i-1), i > 0$,
and $\tau_i$ is in bidegree $(2\ell^i-1, \ell^i-1), i \geq 0$.
We can write this as $\A_\star = H_\star \tensor_{\Z/\ell} \A_\star^{\Top}$,
where $\A_\star^{\Top}$ is the topological dual Steenrod algebra (described by Milnor \cite{Milnor:Steenrod}), but bigraded such that $\xi_i$ has the bidegree above.
The coproduct is then the same as in topology, but the left and right units can be highly nontrivial.
That is, for any $x \in H_{p,q}$ we have $(\eta_L - \eta_R)(x) = y\tau_0$ for $\beta(x) = y \in H_{p-1,q}$.

\subsection*{The motivic Adams spectral sequence}
The mod $\ell$ motivic Adams spectral sequence is constructed in the same way as in topology using motivic Adams resolutions.
For an $\A_\star$-good spectrum $X$ the $E_2$-page is
\[
E_2^{s,t,w} = \Ext^{s,(t,w)}_{\A_\star}(H_\star, H_\star(X))
\]
with $d_r$-differential $d_r: E_r^{s,t,w} \to E_r^{s+r,t+r-1,w}$ \cite[Proposition 7.10]{DI:adams}.
Here $s$ is the filtration, $t$ is the topological degree and $w$ is the motivic weight.
For formal reasons the spectral sequence is conditionally convergent \cite[Definition 5.10]{Boardman} to the $H$-completion
$
\pi_{t-s,w}(X_{H}^\wedge).
$
For $X$ a motivic spectrum of bounded cellular type over fields of finite virtual cohomological dimension at the prime $2$,
and finite cohomological dimension at odd primes, Hu, Kriz and Ormsby proved that 
$X_{H}^\wedge \simeq X_{\ell, \eta}^{\wedge}$ \cite{HKO:adams}.
More recently Mantovani gave a different proof and showed that the assumptions on cellularity and the (virtual)
cohomological dimension can be dropped \cite{Mantovani}; that is %
there is a weak equivalence $X_{H}^{\wedge} \simeq X_{\ell,\eta}^{\wedge}$
for $X$ a connective motivic spectrum over a perfect field \cite[Theorem 1.0.1]{Mantovani}.

There is also a cohomological motivic Adams spectral sequence with $E_2$-page
\[
E_2^{s,t,w} = \Ext^{s,(t,w)}_{\A^\star}(H^\star(X), H^\star).
\]
For $\A_\star$-good spectra for which $H_\star(X) \to \Hom(H^\star(X), H^\star)$ is an isomorphism, and $H_\star(X)$ is of motivic finite type and free over $H_\star$, we have a dualization isomorphism
\cite[Lemma 7.13]{DI:adams}
\[
\Ext^{s,(t,w)}_{\A_\star}(H_\star, H_\star(X))
\cong
\Ext^{s,(t,w)}_{\A^\star}(H^\star(X), H^\star).
\]
In particular this is true for $\S$, $H$ and $\HZ$.
In \Cref{lem:dualization} we prove a dualization isomorphism for some spectra which are not free over $H_\star$.

\subsection*{The $\Ext$-group and the cobar complex}
The $\Ext$-group which appears on the $E_2$-page of the motivic Adams
spectral sequence can be computed with the cobar complex
$C^\bullet(H_\star, H_\star(X))$, which takes the following form
(see Ravenel \cite[Proposition 3.1.2]{Ravenel})
\begin{align}
\label{eq:cobar}
&H_\star \tensor_{H_\star} H_\star(X)
\to
H_\star \tensor_{H_\star} \oA_\star \tensor_{H_\star} H_\star(X)
\to \\
&H_\star \tensor_{H_\star} \oA_\star^{\tensor 2} \tensor_{H_\star} H_\star(X)
\to
\dots
\to
H_\star \tensor_{H_\star} \oA_\star^{\tensor i} \tensor_{H_\star} H_\star(X)
\to \dots. \nonumber
\end{align}
Here $\oA_\star = \ker(\epsilon)$ is the augmentation ideal.
Elements of the cobar complex are commonly written as sums of elements 
$
l[\gamma_1\vert \dots \vert \gamma_s]m
$
where $l \in H_\star, m \in H_\star(X)$ and $\gamma_i \in \oA_\star$.
By $H_\star$-linearity we may assume each $\gamma_i$ is a monomial in $\xi_i$'s and $\tau_i$'s.
If $X = \S$ we may further assume $m = 1$,
in which case it is frequently omitted from the notation.
The differentials are as follows. 
\begin{align}
\label{eq:cobar-diff}
d_s(l[\gamma_1\vert \dots \vert \gamma_s]m) =&
1[\eta_R(l)\vert\gamma_1\vert \dots \vert \gamma_s]m
+\sum_{i=1}^s (-1)^i l[\gamma_1\vert \dots \vert \gamma_{i-1} \vert \Delta(\gamma_i) \vert \gamma_{i+1}\vert \dots \vert \gamma_s]m \\
& + (-1)^{s+1}l[\gamma_1\vert \dots \vert \gamma_s]\psi_X(m)
\nonumber
\end{align}
Note that when $X = \S$ the coaction is given by $\psi_X = \eta_L \tensor 1$.
We say that an element of $C^s(H_\star, H_\star)^{(t,w)}$ or $\Ext^{s,(t,w)}_{\A_\star}(H_\star, H_\star)$ is in Milnor-Witt degree $m = t - s - w$ \cite[p.~2]{DI:Real}.
The cobar differential preserves the internal grading $(t, w)$ and increments $s$ by $1$.
Hence, the cobar differential decreases the Milnor-Witt degree by 1.
Note that the cobar complex is a noncommutative differential graded algebra when $\rho \neq 0$.
For instance mod $2$ we have $\tau\cdot[\tau_0] + [\tau_0]\cdot\tau = \rho[\tau_0]^2$ and $d(\tau^2) \neq 0$.
For odd primes let $C^\bullet_\Top$ be the topological cobar complex \cite[A1.2.11]{Ravenel},
considered as a trigraded object, for example, $[\xi_i]$ has tridegree $(1, (2\ell^i - 2, \ell^i - 2))$.
Denote the homology of $C^{\bullet}_{\Top}$ by $\Ext^{*,(\star)}_{\Top}$.

For the topological $\Ext$-groups Adams proved \cite[Theorem 1]{Adams:finite}
\begin{equation}
\label{eq:adams-iso}
\Ext^{s,t}_{\Top}(\Z/\ell, H_*(\S;\Z/\ell))
\to 
\Ext^{s,t}_{\Top}(\Z/\ell, H_*(\HZ;\Z/\ell))
\end{equation}
is an isomorphism for $t - s < (2\ell - 3)s$.
For $t-s > 0$, the target is zero \cite[Corollary 2]{Adams:finite}.
Over the real numbers Guillou and Isaksen proved an analogue mod 2:
\begin{lemma}[\protect{\cite[Lemma 5.1]{GI:Real}}]
\label{lem:GI}
For all $m > 0$,
$t - s - w = m > 0$, $t - s \neq 0$, $s > \frac{1}{2}(m + 3)$ and $2s > t + 1$
we have the vanishing
\[
\Ext^{s,(t,w)}_{\A^\R_\star}(H_\star^\R, H_\star^\R) = 0.
\]
\end{lemma}
\begin{proof}
By %
\cite[Theorem 1.1]{Adams:period}
$s > \frac{1}{2}(m+3)$ implies vanishing of the topological $\Ext^{m+s,s}_{\Top}(\Z/2, \Z/2)$.
Inspection of the proof of \cite[Lemma 5.1]{GI:Real} proves the statement.
\end{proof}

We will use these vanishing results below in \Cref{sec:detour} and \Cref{sec:odd} to obtain vanishing regions in $\Ext$-groups over general fields.

\section{A detour to homotopy modules}
\label{sec:detour}
In this section we compute the mod $2$ motivic cohomology of the generalized motivic cohomology spectra of Bachmann \cite{Bachmann:veff}.
This allows us to formulate the vanishing results in \Cref{sec:2} as an isomorphism in high filtration between the $\Ext$-groups of $\S$ and $\wHZ$.
The $\Ext$-group of $\wHZ$ can be computed explicitly, see \Cref{cor:Ext-wHZ}.
First we establish some cofiber sequences relating the various spectra and associated homotopy modules. The long exact sequences obtained from the cofiber sequences degenerate to short exact sequences, and it is relatively easy to get complete descriptions of the mod $2$ motivic cohomology of the spectra.
We calculate the cohomological $\Ext$-groups of the spectra
and part of the $\Ext_{\A^\star}(H^\star, H^\star)$-module structure.
Finally we prove a dualization theorem which relates the cohomological $\Ext$-groups to the homological $\Ext$-groups. This is a little involved since some of the spectra have motivic cohomology groups which are $H^\star/\tau$-modules. In particular they are not free $H^\star$-modules, and we have to dualize using $\Ext^1_{H^\star}(-, H^\star)$, since $\Hom_{H^\star}(-, H^\star) = 0$ on these spectra.

Recall from \cite{Bachmann:veff}
that for a perfect field $F$ there are $t$-structures on $\SH(F)$ and $\SH(F)^{\eff}$ defined as
\begin{align*}
\SH(F)_{\geq 0} &= \{ X \in \SH(F) \vsep \ul{\pi}_{i}(X)_* = 0, i < 0 \} \\
\SH(F)_{\leq 0} &= \{ X \in \SH(F) \vsep \ul{\pi}_{i}(X)_* = 0, i > 0 \} \\
\SH(F)_{\geq 0}^{\eff} &= \{ X \in \SH(F)^\eff \vsep \ul{\pi}_{i}(X)_0 = 0, i < 0 \} \\
\SH(F)_{\leq 0}^{\eff} &= \{ X \in \SH(F)^\eff \vsep \ul{\pi}_{i}(X)_0 = 0, i > 0 \}.
\end{align*}
Note that the complex realization of $\SH(\C)_{\geq d}$ is $\SH$,
while the complex realization of $\SH(\C)_{\geq d}^\eff$ is $\SH_{\geq 2d}$.
The hearts of these $t$-structures are identified with homotopy modules and effective homotopy modules, respectively.
Of particular interest are the effective homotopy modules %
\[
\ul{K}_{*}^{MW},\ \ul{K}_{*}^M = \ul{K}_{*}^{MW}/\eta,\ \ul{K}_{*}^{W} = \ul{K}_{*}^{MW}/h,\ \ul{K}_{*}^M/2.
\]
These are the Milnor-Witt K-theory, Milnor $K$-theory, Witt $K$-theory and mod 2 Milnor $K$-theory sheaves, respectively, cf.~\cite{Morel:A1}, %
\cite{Morel:witt}.
The elements $\eta$ and $h$ correspond to the Hopf map and the hyperbolic plane in $\ul{K}^{MW}_{-1} = \W$ and $\ul{K}^{MW}_0 \cong \GW$, respectively.
By \cite[Lemma 6]{Bachmann:veff} it does not matter where the cokernels/cones of these elements are formed.
Note that both $h$ and $\eta$ induce the zero map on mod 2 motivic cohomology.
By taking effective covers of the above homotopy modules we get generalized motivic cohomology theories as observed by Bachmann \cite[p.~15]{Bachmann:veff}:
\[
\wHZ := \f_0(\ul{K}_{*}^{MW}),\ \HZ = \f_0(\ul{K}_{*}^M),\ \HW := \f_0(\ul{K}_{*}^{W}),\ H =\f_0(\ul{K}_{*}^M/2).
\]
Here $\HZ$ represents motivic cohomology, while $\wHZ$ represents Milnor-Witt motivic cohomology \cite{Bachmann-Fasel}.
Bachmann computed the coefficients of $\wHZ$ and $\HW$ to be \cite[Theorem 17]{Bachmann:veff}
\[
\pi_{t,w}(\wHZ) = \begin{cases}
K^{MW}_{-t}(F) & t - w = 0 \\
H^{-t,-w}(F;\Z) & t - w \neq 0,
\end{cases}
\qquad
\pi_{t,w}(\HW) = \begin{cases}
W(F) & t - w = 0, t \geq 0 \\
I(F)^{-t} & t - w = 0, t < 0 \\
H^{-t,-w}(F;\Z/2) & t - w \neq 0.
\end{cases}
\]

In \cite{Bachmann:veff} the base field is assumed to be perfect for the $t$-structure on $\SH(F)$ to be well behaved. Hence, most of the computations in this section require the base field to be perfect.
However, for the convergence results we only need to make a comparison between the $\Ext$-group of the motivic sphere spectrum and $\wHZ$,
and this is an entirely algebraic statement.
Hence perfectness of the base field is not needed for the convergence results in \Cref{sec:strong}, unless explicitly required.
\begin{lemma}
\label{lem:cofibs}
Over a perfect field of characteristic not 2, there are cofiber sequences of the following form.
\begin{align}
&\HZ \xrightarrow{h} \wHZ \to \HW \label{eq:C1} \\
&\HZ \xrightarrow{2} \HZ \to H \label{eq:C2} \\
&\Sigma^{1,1}\ul{K}_{*}^{W} \xrightarrow{\eta} \ul{K}_{*}^{W} \to \ul{K}_{*}^M/2 \label{eq:C3} \\
&\HW_{\geq 1} \to \HW \to \ul{K}_{*}^{W} \label{eq:C4} \\
&H_{\geq 1} \to H \to \ul{K}_{*}^{M}/2 \label{eq:C5}
\end{align}
There are also the equivalences below.
\begin{align}
\HW_{\geq 1} &\xrightarrow{\simeq} H_{\geq 1} \label{eq:eq1} \\
\Sigma^{0,-1}H &\xrightarrow{\simeq} H_{\geq 1} \label{eq:eq2}
\end{align}
\end{lemma}
\begin{proof}
The cofiber sequence \eqref{eq:C1} is part of the proof of \cite[Lemma 19]{Bachmann:veff}.
The cofiber sequences \eqref{eq:C2}, \eqref{eq:C4} and \eqref{eq:C5} are by definition.
The cofiber sequence \eqref{eq:C3} is the short exact sequence in $\SH(k)^{\heart}$
\[
0 \to \Sigma^{1,1}\ul{K}_{*}^{W} \xrightarrow{\eta} \ul{K}_{*}^{W} \to \ul{K}_{*}^M/2 \to 0.
\]
This sequence is exact since $K^W_*(F) \cong I^*(F)$ for a field $F$ \cite[Theoreme 2.1]{Morel:witt} (see also \cite[(2.2), p.~66]{Morel:A1}) and since multiplication by $\eta$ corresponds to the inclusion of $I^{n+1}(F)$ into $I^n(F)$ \cite[p.~692]{Morel:witt}.
The equivalence \eqref{eq:eq1} is \cite[Theorem 17]{Bachmann:veff}.
The equivalence \eqref{eq:eq2} follows from the equivalence of cofiber sequences
\[
\begin{tikzcd}
\Sigma^{0,-1}H \ar[r, "\tau"]\ar[d] & H \ar[r]\ar[d] & H/\tau\ar[d] \\
H_{\geq 1} \ar[r] & H \ar[r] & \ul{K}_{*}^M/2.
\end{tikzcd}
\]
Here $H/\tau \to \ul{K}_{*}^M/2$ is an equivalence since when evaluated on a field the source and target are both mod 2 Milnor $K$-theory which is mapped by the identity, cf.~\cite[2.2]{Hoyois:fromto}.
\end{proof}

Next we compute mod 2 motivic cohomology of $\wHZ, \HW, \ul{K}_{*}^W, \ul{K}_{*}^M/2$,
the $\Ext$-groups of the motivic cohomology, %
and their relation to $\Ext_{\A_\star}(H_\star, H_\star)$.
Throughout we use the following convention for suspensions and bigraded hom-groups:
For left $\A^\star$-modules we set $(\Sigma^{t,w}M^\star)^{p,q} = M^{p-t,q-w}$ and
$\Hom^{(t,w)}_{\A^\star}(M^\star, N^\star) = \Hom_{\A^\star}(M^\star, \Sigma^{t,w}N^\star)$.
For left $\A_\star$-comodules we set $(\Sigma^{t,w}M_\star)_{p,q} = M_{p-t,q-w}$ and
$\Hom^{(t,w)}_{\A_\star}(M_\star, N_\star) = \Hom_{\A_\star}(M_\star, \Sigma^{-t,-w}N_\star)$.
With this convention
$H^\star(\Sigma^{t,w}X) = \Sigma^{t,w}H^\star(X)$,
$H_\star(\Sigma^{t,w}X) = \Sigma^{t,w}H_\star(X)$, $\Hom_{\A^\star}^{(t,w)}(\A^\star, H^\star) = H_{t,w} = \Hom_{\A_\star}^{(t,w)}(\A_\star, H_\star)$,
and suspensions on the interesting variable of the $\Ext$-groups can be moved inside and outside without changing the signs of the suspensions.
The $\Hom$-groups are considered as homological objects although they are indexed with superscripts.

\begin{lemma}
\label{lem:ell-KW}
Over a perfect field of characteristic not 2 and for an odd prime $\ell$, the motivic cohomology groups $H^\star(\ul{K}_{*}^{W};\Z/\ell)$ and  $H^\star(\HW;\Z/\ell)$ are trivial.
\end{lemma}
\begin{proof}
The first claim follows from the commutative diagram below with cofiber sequences as columns and rows.
\[
\begin{tikzcd}
\Sigma^{1,1}\ul{K}_{*}^W \ar[r, "\eta"]\ar[d, "\ell"] & \ul{K}_{*}^W \ar[r]\ar[d, "\ell"] & \ul{K}_{*}^M/2\ar[d, "\ell"] \\
\Sigma^{1,1}\ul{K}_{*}^W \ar[r, "\eta"]\ar[d] & \ul{K}_{*}^W \ar[r]\ar[d] & \ul{K}_{*}^M/2 \ar[d] \\
\Sigma^{1,1}\ul{K}_{*}^W/\ell \ar[r, "\eta"] & \ul{K}_{*}^W/\ell \ar[r] & \ul{K}_{*}^M/(2,\ell) = 0
\end{tikzcd}
\]
It follows that $\eta: \Sigma^{1,1}\ul{K}_{*}^W/\ell \to \ul{K}_{*}^W/\ell$ is an equivalence.
But since $H^\star(\eta;\Z/\ell) = 0$ we must have $H^\star(\ul{K}_{*}^W;\Z/\ell) = 0$.

The second vanishing claim is deduced from the cofiber sequence \eqref{eq:C4} and the equivalence \eqref{eq:eq1}.
\end{proof}

Recall that $\A^\star(1)$ is the $H^\star$-subalgebra of $\A^\star$ generated by $\Sq^1$ and $\Sq^2$.
The next lemma allows us to check exactness of short exact sequences in the finite subalgebra $\A^\star(1)$ of $\A^\star$.
\begin{lemma}[\protect{\cite[Lemma 3.2.15]{Gregersen}}]
\label{lem:A1-free}
$\A^\star$ is free as a right $\A^\star(1)$-module.
\end{lemma}
\begin{proof}
In \cite[Lemma 3.2.15]{Gregersen} it is proven that  $\A^\star$ is free as a left $\A^\star(1)$-module.
Since the conjugation $c : \A_\star \to \A_\star$ \cite[p.~3845]{HKO:steenrod}
satisfies $t (c \tensor c)\Delta = \Delta c$,
where $t$ is the twist, the same proof using the conjugate monomial basis shows that $\A^\star$ is free as a right $\A^\star(1)$-module.
\end{proof}

\begin{lemma}
\label{lem:A-ses}
We have a short exact sequence of left $\A^\star$-modules
\begin{align*}
0\to \Sigma^{2,1}\A^\star/\A^\star (\tau, \Sq^2 + \rho\Sq^1) &\xrightarrow{\cdot(\Sq^2 + \rho\Sq^1)}
\A^\star/\A^\star \tau \\
&\to \A^\star/\A^\star (\tau, \Sq^2 + \rho\Sq^1) \to 0.
\end{align*}
\end{lemma}
\begin{proof}
By \Cref{lem:A1-free} $\A^\star$ is free as a right $\A^\star(1)$-module.
Hence it suffices to show exactness of the corresponding $\A^\star(1)$-modules and tensor up with $\A^\star$.
Exactness is easy to check with help of the following figure of $\A^\star(1)$:
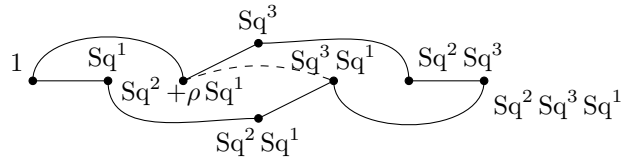
\[
\begin{tikzpicture}
  \draw[fill] (0, 0) circle [radius=0.05];
  \draw[fill] (1, 0) circle [radius=0.05];
  \draw[fill] (2, 0) circle [radius=0.05];
  \draw[fill] (3, 0.5) circle [radius=0.05];
  \draw[fill] (3, -0.5) circle [radius=0.05];
  \draw[fill] (4, 0) circle [radius=0.05];
  \draw[fill] (5, 0) circle [radius=0.05];
  \draw[fill] (6, 0) circle [radius=0.05];

  {\node[above left=0pt] at (0,0) {$1$};}
  {\node[above=0pt] at (1,0) {$\Sq^1$};}
  {\node[below=-5pt] at (2,0) {$\Sq^2 + \rho\Sq^1$};}
  {\node[above=0pt] at (3,0.5) {$\Sq^3$};}
  {\node[below=0pt] at (3,-0.5) {$\Sq^2\Sq^1$};}
  {\node[above=-1pt] at (4,0) {$\Sq^3\Sq^1$};}
  {\node[above right=0pt] at (5,0) {$\Sq^2\Sq^3$};}
  {\node[below right=0pt] at (6,0) {$\Sq^2\Sq^3\Sq^1$};}

  \draw (0, 0) --(1, 0);
  \draw[] (0, 0) to [out=90, in=90](2, 0);
  \draw (2, 0) --(3, 0.5);
  \draw[] (1, 0) to [out=-90, in=180](3, -0.5);
  \draw[dashed] (2, 0) to [out=20, in=160](4, 0);
  \draw (3, -0.5) --(4, 0);
  \draw[] (3, 0.5) to[out=0, in=90](5, 0);
  \draw[] (4, 0) to [out=-90, in=-90](6, 0);
  \draw[] (5, 0) --(6, 0);
\end{tikzpicture}
\]
In the figure left multiplications by $\Sq^1$ are indicated by straight lines,
while left multiplications by $\Sq^2 + \rho\Sq^1$ are indicated by curved lines.
The dashed curve indicates the relation
\[
(\Sq^2 + \rho\Sq^1)^2 = \rho \Sq^{2} \Sq^{1} + \rho \Sq^{3} + \tau \Sq^{3} \Sq^{1}  = \Sq^3\Sq^1\tau.
\]
Note that 
\begin{equation}
\label{eq:Sq2tau}
\Sq^2\tau = \tau(\Sq^2 + \rho\Sq^1),
\end{equation}
so multiplication by $\Sq^2 + \rho\Sq^1$ is well defined on $\A^\star/\tau$.
Hence the kernel of right multiplication by $\Sq^2 + \rho\Sq^1$ on $\A^\star/\tau$ is generated as a left $\A^\star$-module by
$
\{\Sq^2 + \rho\Sq^1, \Sq^3, \Sq^2\Sq^3, \Sq^2\Sq^3\Sq^1 \},
$
that is, the image of right multiplication by $\Sq^2 + \rho\Sq^1$.
\end{proof}

\begin{lemma}
\label{lem:HKW}
Over a perfect field of characteristic not 2, the mod 2 motivic cohomology groups of the homotopy modules $\ul{K}_{*}^M/2$ and $\ul{K}_{*}^W$ are
\begin{align}
  H^\star(\ul{K}_{*}^M/2) &\cong \Sigma^{1,-1}A^\star/\A^\star\tau \label{eq:HKM2} \\
  H^\star(\ul{K}_{*}^W) &\cong \Sigma^{1,-1}A^\star/\A^\star(\tau, \Sq^2 + \rho\Sq^1). \label{eq:HKW}
\end{align}
\end{lemma}
\begin{proof}
From \eqref{eq:C5} and \eqref{eq:eq2} we obtain the following long exact sequence.
\[
\Sigma^{1,0}\A^\star \xrightarrow{\cdot \tau} \Sigma^{1,-1}\A^\star \to
H^\star(\ul{K}_{*}^M/2) \to \A^\star \xrightarrow{\cdot \tau} \Sigma^{0,-1}\A^\star
\]
Since $\tau$ is not a zero divisor this implies \eqref{eq:HKM2}.

The second claim requires more work. 
First consider the case when the base field is $\R$, so that $\A^{2,1} \cong \Z/2\{\rho\Sq^1, \Sq^2\}$. 
The cofiber sequence \eqref{eq:C3} gives rise to the following map of short exact sequences of left $\A^\star$-modules (since $H^\star(\eta) = 0$).
\[
\begin{tikzcd}[column sep=huge]
\Sigma^{2,1}H^\star(\ul{K}_{*}^{W}) \ar[r, rightarrowtail]\ar[d, "\Sigma^{2,1}f"] & H^\star(\ul{K}_{*}^M/2) \ar[r, two heads]\ar[d, "\cong"] & H^\star(\ul{K}_{*}^{W}) \ar[d, "f"] \\
\Sigma^{3,0}\A^\star/(\tau, \Sq^2 + \rho\Sq^1) \ar[r, rightarrowtail, "\cdot(\Sq^2 + \rho\Sq^1)"] & \Sigma^{1,-1}\A^\star/\tau \ar[r, two heads] & \Sigma^{1,-1}\A^\star/(\tau, \Sq^2 + \rho\Sq^1)
\end{tikzcd}
\]
The induced composite
$g: H^\star(\ul{K}_{*}^M/2) \to H^\star(\ul{K}_{*}^W) \to \Sigma^{-2,-1}H^\star(\ul{K}_{*}^M/2)$
is right multiplication by $\Sq^2 + \rho\Sq^1$.
Since $H^{-1,-2}(\ul{K}_*^W) = 0$, the composite $g$ cannot be zero in bidegree $(1,-1)$. So $g$ must send $1$ to $\Sq^2, \Sq^2 + \rho\Sq^1$ or $\rho\Sq^1$ over $\R$.
The map $g$ cannot be multiplication by $\rho\Sq^1$ since changing the base to $\C$ would force $g$ to be zero.
If $g$ sends $1$ to $\Sq^2$ then $\tau \in \A^\star \tau$ maps to $\tau\Sq^2 \not\in \A^\star\tau$ if $\rho \neq 0$. %
That is, $g$ would not be a left $\A^\star$-module map. So $g$ is forced to be $\cdot (\Sq^2 + \rho \Sq^1)$; hence the left square commutes. 
Inductively this implies that $f$ is an isomorphism and \eqref{eq:HKW}.

A similar argument also works over $\Q$ and finite fields of odd characteristic. 
The map $g$ will be right multiplication by $\Sq^2 + a\Sq^1$, where $a \in K^M_1/2$. 
That $g$ must be an $\A^\star$-module homomorphism and $g\circ g = 0$ forces $a = \rho$. 
The result for a general field is now obtained by using base change from a prime field. 
\end{proof}

The cofiber sequences of \Cref{lem:cofibs} give us short exact sequences in motivic cohomology. Hence, we obtain calculations of the mod 2 motivic cohomology of the spectra.
\begin{lemma}
\label{lem:sess}
Over a perfect field of characteristic not $2$
there are short exact sequences in mod 2 motivic cohomology
\begin{align}
  &0 \to \Sigma^{1,0}H^\star(\HZ) \to \A^\star \to H^\star(\HZ) \to 0 \label{eq:ses5} \\
  &0 \to \Sigma^{1,0}\A^\star \xrightarrow{\cdot\tau} \Sigma^{1,-1}\A^\star \to H^\star(\ul{K}_{*}^{M}/2) \to 0 \label{eq:ses1} \\
  &0 \to \Sigma^{2,1}H^\star(\ul{K}_{*}^W) \to H^\star(\ul{K}_{*}^M/2) \to H^\star(\ul{K}_{*}^W) \to 0 \label{eq:ses2}\\
  &0 \to H^\star(\HW) \to \Sigma^{0,-1}\A^\star \to \Sigma^{-1,0}H^\star(\ul{K}_{*}^W) \to 0 \label{eq:ses3}\\
  &0 \to \Sigma^{1,0}H^\star(\HZ) \to H^\star(\HW) \to H^\star(\wHZ) \to 0. \label{eq:ses4}
\end{align}
As short exact sequences of $H^\star$-modules, \eqref{eq:ses5}, \eqref{eq:ses2} and \eqref{eq:ses4} are split. %
\end{lemma}
\begin{proof}
The short exact sequences \eqref{eq:ses5}, \eqref{eq:ses1}, \eqref{eq:ses2} and \eqref{eq:ses4} are just restatements of \eqref{eq:C2}, \eqref{eq:HKM2}, \eqref{eq:C3} and \eqref{eq:C1}, since $H^\star(h) = 0 = H^\star(\eta)$.

The cofiber sequences \eqref{eq:C4} and \eqref{eq:C5} combined with the equivalences \eqref{eq:eq1} and \eqref{eq:eq2} induce the commutative diagram
\[
\begin{tikzcd}
\Sigma^{1,0}\A^\star \ar[r, "\cdot \tau"]\ar[d] & \Sigma^{1,-1}\A^\star \ar[r, two heads]\ar[d, "="] & H^\star(\ul{K}_{*}^M/2) \ar[r]\ar[d, two heads] & \A^\star \ar[d] \\
\Sigma^{1,0}H^\star(\HW) \ar[r] & \Sigma^{1,-1}\A^\star \ar[r] & H^\star(\ul{K}_{*}^W) \ar[r] & H^\star(\HW),
\end{tikzcd}
\]
where the map $H^\star(\ul{K}_{*}^M/2) \twoheadrightarrow H^\star(\ul{K}_{*}^W)$ is surjective by \eqref{eq:HKW}.
Hence the map $\Sigma^{1,-1}\A^\star \to H^\star(\ul{K}_{*}^W)$ is surjective, so \eqref{eq:ses3} is exact.
\end{proof}

\begin{corollary}
\label{cor:tom}
The mod 2 motivic cohomology of the Witt-motivic cohomology spectrum is
\[
H^\star(\HW) \cong \Sigma^{0,-1}\A^\star(\tau, \Sq^2 + \rho\Sq^1).
\]
\end{corollary}
\begin{proof}
This follows by combining the short exact sequence \eqref{eq:ses3} and the isomorphism \eqref{eq:HKW}.
\end{proof}
This has been independently computed by Tom Bachmann \cite{Bachmann:private} to be
\[
H^\star(\HW) = \coker\left(
\begin{pmatrix}
\cdot\Sq^2 & \cdot\Sq^3\Sq^1 \\
\cdot \tau & \cdot(\Sq^2+\rho\Sq^1)
\end{pmatrix}
: \Sigma^{2,1}\A^\star \oplus \Sigma^{4,1}\A^\star \to A^\star \oplus \Sigma^{2,0}\A^\star\right).
\]
These are the same since we have the exact sequence
\begin{align}
\label{eq:tom}
\Sigma^{2,2}\A^\star(1) \oplus \Sigma^{4,2}\A^\star(1)
&\xrightarrow{
\begin{pmatrix}
\cdot\Sq^2 & \cdot\Sq^3\Sq^1 \\
\cdot\tau & \cdot(\Sq^2+\rho\Sq^1)
\end{pmatrix}} \\
&\to \Sigma^{0,1}\A^\star(1) \oplus \Sigma^{2,1}\A^\star(1)
\xrightarrow{(\cdot\tau, \cdot(\Sq^2 + \rho\Sq^1))}
\A^\star(1)(\tau, \Sq^2 + \rho\Sq^1)
\to 0 \nonumber
\end{align}
Exactness is checked using the figure in \Cref{lem:A-ses},
and then we tensor up with $\A^\star$ to see that the expressions for $H^\star(\HW)$ agree.

Next we compute the $\Ext$-groups of the generalized motivic cohomology spectra and the associated homotopy modules.
\begin{lemma}
\label{lem:modtau-ext}
The $\Ext$-groups of $\A^\star/\A^\star\tau$ and $\A^\star/\A^\star(\tau, \Sq^2 + \rho\Sq^1)$ are
\begin{equation}
\Ext^{s}_{\A^\star}(\A^\star/\A^\star\tau, H^\star) \cong
\begin{cases}
\Sigma^{0,1}H_\star/\tau & s = 1 \\
0 & \text{otherwise,}
\end{cases} \label{eq:Ext1} \\
\end{equation}
and
\begin{equation}
\Ext^{s+1}_{\A^\star}(\A^\star/\A^\star(\tau, \Sq^2 + \rho\Sq^1), H^\star)
\cong
\begin{cases}
\Sigma^{2s,s+1}H_\star/\tau & s \geq 0 \\
0 & \text{otherwise.}
\end{cases}
\label{eq:Ext2}
\end{equation}
\end{lemma}
\begin{proof}
To show \eqref{eq:Ext1} use the resolution $0 \to \Sigma^{1,0}\A^\star \xrightarrow{\cdot \tau} \A^\star \to \A^\star/\tau \to 0.$
For \eqref{eq:Ext2} use the long exact sequence of Ext groups associated to the short exact sequence of \Cref{lem:A-ses}.
\end{proof}

Since $\A^\star/\tau$ is not free as an $H^\star$-module we must work a bit to compute the $\Ext_{\A^\star}(H^\star, H^\star)$-module structure.
\begin{remark}
\label{rmk:products}
Recall that for any $\A^\star$-module $M$, the Yoneda product gives $\Ext_{\A^\star}(M,H^\star)$ the structure of a right $\Ext_{\A^\star}(H^\star, H^\star)$-module as follows \cite[pp.~47-48]{Rognes:adams}.
Let $P^\bullet \to M$ and $Q^\bullet \to H^\star$ be projective $\A^\star$-resolutions of $M$ and $H^\star$ respectively,
and consider cocycles $g \in \Hom(P^u, H^\star)$ and $f \in \Hom(Q^t, H^\star)$. The Yoneda product of $g$ and $f$ is the composite $fg^t$, where $g^t$ is obtained by inductively finding lifts $g^i$ in diagram \eqref{eq:Yoneda-product}.
\begin{equation}
\label{eq:Yoneda-product}
\begin{tikzcd}
& H^\star \\
\dots \ar[r] & Q^t \ar[r]\ar[u, "f"] & \dots \ar[r] & Q^1 \ar[r] & Q^0 \ar[r] & H^\star \\
\dots \ar[r] & P^{u+t} \ar[r]\ar[u, "g^t"] & \dots \ar[r] & P^{u+1} \ar[r]\ar[u, "g^1"] & \ar[u, "g^0"] P^{u} \ar[ru, "g" below]
\end{tikzcd}
\end{equation}
In \Cref{lem:modtau-ext} we will calculate the Yoneda product using resolutions $P^\bullet \to M$ which are not necessarily projective, but which satisfy $\Ext^{i}_{\A^\star}(P^j, N) \neq 0$ only if $i = s$ for some $s$ for all left $\A^\star$-modules $N$, and $\Ext^{i}_{\A^\star}(P^j\to P^{j-1}, N) = 0$ for all $j$.
To keep the notation simple we momentarily step into the derived category of $\A^\star$-modules.
This means that we have yet another shift operation $[1]$
such that $\Ext^{i}_{\A^\star}(M, N) = \RHom_{\mathcal{D}(\A^\star)}(M, N[i])$.
Then the Yoneda product is formed in the same way, using $P^i[-s]$ in place of $P^i$ in \eqref{eq:Yoneda-product}.
\end{remark}

\begin{lemma}
Let $MW^{\geq 1}$ denote the right ideal in $\Ext_{\A^\star}(H^\star,H^\star)$ generated by those classes in positive Milnor-Witt degree.
The isomorphism \eqref{eq:Ext2} can be enhanced to an $\Ext_{\A^\star}(H^\star,H^\star)$-module isomorphism
\[
\Ext_{\A^\star}(\A^\star/\A^\star(\tau, \Sq^2 + \rho\Sq^1), H^\star)
\cong
\{a\}\Ext_{\A^\star}(H^\star, H^\star)/(MW^{\geq 1}, h_0)
\]
where $\vert a \vert = (1, (0, 1))$. 
Note that $\Ext^{s+1}_{\A^\star}(\A^\star/\A^\star(\tau, \Sq^2 + \rho\Sq^1), H^\star)$ is generated over $H_\star/\tau$ by $a\cdot h_1^s$. Note that in the cobar complex $h_0$ is represented by $[\tau_0]$ and $h_1$ is represented by $[\xi_1]$.
\end{lemma}

\begin{proof}
To establish the $\Ext_{\A^\star}(H^\star, H^\star)$-module structure, we use the following resolution for calculating Yoneda products.
\begin{equation}
\label{eq:A/tau-res}
\dots \to
\Sigma^{2,1}\A^\star/\tau \xrightarrow{\cdot(\Sq^2 + \rho\Sq^1)}
\Sigma^{2,1}\A^\star/\tau \xrightarrow{\cdot(\Sq^2 + \rho\Sq^1)}
\A^\star/\tau \to
\A^\star/(\tau, \Sq^2 + \rho\Sq^1) \to 0.
\end{equation}
This is a resolution in the sense that $\Ext^{s}_{\A_\star}(\A^\star/\tau, H^\star) = 0$ if $s \neq 1$,
and $(\Sq^2 + \rho\Sq^1)^* = 0$ on $\Ext^1_{\A_\star}(\A^\star/\tau, H^\star)$.
A dimension shifting argument similar to \cite[Exercise 2.4.3]{Weibel} shows that the higher $\Ext$-groups may be calculated with this resolution.

We now show that $ah_1^s$ is nontrivial for $s \geq 1$. 
As observed in \Cref{rmk:products} we can use the resolution \eqref{eq:A/tau-res} to compute the Yoneda product.
Consider the product of the representative of $ah_1^s$ in $\Ext^{1,(2s,s)}_{\A^\star}(\Sigma^{2s,s}\A^\star/\tau,H^\star)$ with the representative of $h_1$ in $\Ext^{0,(2,1)}_{\A^\star}(Q^1, H^\star)$.
Consider a resolution of $H^\star$ beginning with 
\[
\begin{tikzcd}
\cdots \ar[r] & Q^1 = \oplus_{i\in \N} \A^\star\{\eta_i\} \ar[r, "f"] & Q^0 = \A^\star \ar[r] & H^\star
\end{tikzcd}
\]
where $\eta_i$ maps to $\Sq^{2^i}$ under $f$, and $\eta_i$ has bidegree $(2^i, 2^{i-1})$. 
By definition, $h_i$ is the dual of $\eta_i$.
The Yoneda product of $ah_1^{s}$ and $h_1$ is calculated with the $\Sigma^{2s,s}$-suspension of the diagram below.
\[
\begin{tikzcd}[column sep=huge]
\Sigma^{2,1}H^\star \\
Q^1 \supset \Sigma^{2,1}\A^\star \ar[r, "f"]\ar[u, "h_1"] & \A^\star \ar[r] & H^\star \\
\Sigma^{2,1}\A^\star/\tau[-1] \ar[r, "\cdot(\Sq^2 + \rho\Sq^1)"]\ar[u] & \A^\star/\tau[-1]\ar[u] \ar[ur, "a"]
\end{tikzcd}
\]
The maps $\A^\star/\tau[-1] \to \A^\star$ are induced by the image of the identity map by $\Hom_{\A^\star}(\A^\star, \A^\star)\to\Hom_{\A^\star}(\A^\star/\tau[-1], \A^\star)$.
The lifts of $ah_1^{s}$ are canonical, and the commutativity of this diagram follows from the map of distinguished triangles in $\mathcal D(\A^\star)$
\[
\begin{tikzcd}
\Sigma^{2,1}\A^\star/\tau[-1] \ar[r]\ar[d, "\cdot(\Sq^2 + \rho\Sq^1)"] & \Sigma^{2,2}\A^\star \ar[r, "\cdot \tau"]\ar[d, "\cdot\Sq^2"] & \Sigma^{2,1}\A^\star \ar[r]\ar[d, "\cdot(\Sq^2+\rho\Sq^1)"] &  \Sigma^{2,1} \A^\star/\tau \ar[d, "\cdot(\Sq^2 + \rho\Sq^1)"] \\
\A^\star/\tau[-1] \ar[r] & \Sigma^{0,1}\A^\star \ar[r, "\cdot \tau"] & \A^\star \ar[r] & \A^\star/\tau.
\end{tikzcd}
\]
Inductively, we see that $ah_1^s$ is nontrivial for all $s$ and $ah_1^s\cdot h_i = 0$ when $i\neq 1$. 
The products with elements of $MW^{\geq 1}$ vanish for degree reasons.
\end{proof}

\begin{corollary}
\label{cor:Ext-wHZ}
The mod 2 $\Ext$-groups of $\HW$ and $\cHZ$ are
\[
\Ext^{s,(t,w)}_{\A^\star}(H^\star(\HW), H^\star)
\cong \begin{cases}
(\Sigma^{2s,s}H_\star/\tau)_{(t,w)} & s > 0 \\
H_{t,w} & s = 0,
\end{cases}
\]
and
\[
\Ext^{s,(t,w)}_{\A^\star}(H^\star(\wHZ), H^\star)
\cong
\begin{cases}
\begin{aligned}
&\Ext^{s,(t,w)}_{\A^\star}(H^\star(\HW), H^\star) \\
&\qquad\bigoplus \\
&\Ext^{s-1,(t-1,w)}_{\A^\star}(H^\star(\HZ), H^\star)
\end{aligned}
 & s > 0\\
H_{t,w} & s = 0.
\end{cases}
\]
Furthermore $\Ext^{s,(t,w)}_{\A^\star}(H^\star(\HW), H^\star) = 0$ for $t - s - w \neq 0$, $s > 0$,
and $\Ext^{s,(t,w)}_{\A^\star}(H^\star(\HZ), H^\star) = 0$ for $t - s > 0$.
Over $\R$ we have $\Ext^{s,(t,w)}_{\A^\star_{\R}}(H^\star_\R(\HZ), H^\star_\R) = 0$ when $t-s \neq 0$ or $t - s - w$ is odd, and $s > 0$.
\end{corollary}
\begin{proof}
For $s = 0$ we get the extension
\[
0 \to (\Sigma^{0,-1}H_\star)_{(t,w)}
\to \Ext^{0,(t,w)}_{\A^\star}(H^\star(\HW), H^\star)
\to (H_\star/\tau)_{(t,w)} %
\to 0
\]
from \eqref{eq:ses3}.
We compare with mod 2 motivic cohomology using the canonical map $\HW \to H$ to determine that the extension is non-split.
For $s > 0$ we get
\begin{align*}
\Ext^{s,(t,w)}_{\A^\star}(H^\star(\HW), H^\star)
&\cong
\Ext^{s+1,(t,w)}_{\A^\star}(\Sigma^{-1,0}H^\star(\ul{K}_{*}^W), H^\star) \\
&\cong
\Ext^{s+1,(t,w)}_{\A^\star}(\Sigma^{-1,0}\Sigma^{1,-1}\A^\star/(\tau, \Sq^2 + \rho\Sq^1), H^\star) \\
&\cong
(\Sigma^{2s,s}H_\star/\tau)_{(t,w)}.
\end{align*}

From \eqref{eq:ses4} we get the long exact sequence
\begin{align}
\dots \label{eq:cHZ-les}
&\to \Ext^{s-1,(t-1,w)}_{\A^\star}(H^\star(\HZ), H^\star)
\to \Ext^{s,(t,w)}_{\A^\star}(H^\star(\cHZ), H^\star) \\
&\to \Ext^{s,(t,w)}_{\A^\star}(H^\star(\HW), H^\star)
\to \Ext^{s,(t-1,w)}_{\A^\star}(H^\star(\HZ), H^\star)
\to \dots. \nonumber
\end{align}
The composite
\[
\Ext^{s-1,(t-1,w)}_{\A^\star}(H^\star(\HZ), H^\star)
\to
\Ext^{s,(t,w)}_{\A^\star}(H^\star(\cHZ), H^\star)
\to
\Ext^{s,(t,w)}_{\A^\star}(H^\star(\HZ), H^\star)
\]
is an isomorphism by comparing \eqref{eq:C1} and \eqref{eq:C2},
hence \eqref{eq:cHZ-les} is split for $s > 0$,
and
\[
\Ext^{0,(t,w)}_{\A^\star}(H^\star(\cHZ), H^\star) \cong\Ext^{0,(t,w)}_{\A^\star}(H^\star(\HW), H^\star) \cong H_{t,w}.
\]

The statements over $\R$ can be seen by a calculation in the cobar complex, cf.~\cite[Figure 3]{DI:Real}.
\end{proof}

So far in this section we have used cohomological $\Ext$-groups.
In the next lemmas we establish how the cohomological $\Ext$-groups relate to the homological $\Ext$-groups, which we use elsewhere in this paper.
We also check that the spectra defined by Bachmann are $\A_\star$-good, hence that the $\Ext$-groups are the $E_2$-pages of their respective motivic Adams spectral sequences.
When we know that there is a dualization isomorphism between the cohomological and homological $\Ext$-groups we can compare directly with the homological $\Ext$-group of the motivic sphere spectrum. All of this can be avoided, since for the motivic sphere spectrum we know that the $\Ext$-groups satisfy a dualization isomorphism by \cite[Lemma 7.13]{DI:adams}. However, the relation between the homological and cohomological $\Ext$-groups given in \Cref{lem:dualization} seems worth elaborating.

\begin{lemma}
\label{lem:adams-res}
Suppose $\E$ is a spectrum that fits into a cofiber sequence
$\Sigma^{p,q} \E \xrightarrow{f} \E \to \F$, where the map 
\begin{equation}
\label{eq:adams-res-iso}
H_\star(\F)\tensor_{H_\star} H_\star(H^{\wedge s}) \to H_\star(\F \wedge H^{\wedge s})
\end{equation}
is an isomorphism for all $s$ and $H_\star(f) = 0$. Alternatively, suppose $\E$ is a
spectrum that fits into a cofiber sequence $\F' \to \F \to \E$, where
for both $\F$ and $\F'$ the map \eqref{eq:adams-res-iso} is an
isomorphism. Then the map \eqref{eq:adams-res-iso} is an isomorphism for $\E$ as well.
\end{lemma}
\begin{proof}
This is straightforward; use the snake lemma and the 5-lemma.  
\end{proof}

\begin{corollary}
\label{cor:adams-res}
If $\E$ is any one of the spectra $H$, $\ul{K}_{*}^M/2$, $\HZ$,
$\ul{K}_{*}^W$, $\HW$, or $\wHZ$, the map
\begin{equation*}
H_\star(\E)\tensor_{H_\star} H_\star(H^{\wedge s}) \to H_\star(\E \wedge H^{\wedge s})
\end{equation*}
is an isomorphism. 
\end{corollary}

\begin{proof}
We can apply \Cref{lem:adams-res} to these spectra by using the
cofiber sequences from \Cref{lem:cofibs} and \Cref{lem:sess}.
\end{proof}

\begin{remark}
We expect the spectra in \Cref{cor:adams-res} to be cellular, in which
case the isomorphism \eqref{eq:adams-res-iso} follows from \cite[Lemma
  7.6]{DI:adams}.
\end{remark}

Next we establish a dualization lemma.
Note the shift in the (co)homological degree of the $\Ext$-groups for the spectra $\ul{K}_{*}^M/2$ and $\ul{K}_{*}^W$.
This is due to $\Hom_{H^\star}(H^\star(\ul{K}_{*}^M/2), H^\star)$ being zero,
while $\Ext^1_{H^\star}(H^\star(\ul{K}_{*}^M/2), H^\star)$ is nonzero.
From \Cref{lem:dualization} we get that \Cref{cor:Ext-wHZ} is a computation of $\Ext_{\A_\star}(H_\star, H_\star(\wHZ))$.
By \eqref{eq:adams-res-iso} this is the $E_2$-page of the motivic Adams spectral sequence of $\wHZ$.
\begin{lemma}
\label{lem:dualization}
(1) If $X$ is $H$, $\ul{K}_{*}^M/2$, $\HZ$, $\ul{K}_{*}^W$, $\HW$, or
$\wHZ$, then the adjoint of the evaluation map
\begin{equation}
\label{eq:eval}
H\wedge X \to [[X, H], H]_H
\end{equation}
is an equivalence. 
Here $[X,Y]$ is the internal $\Hom$-object in $\SH(F)$ and $[X,Y]_H = \hoeq([X, Y] \rightrightarrows [H\wedge X, Y])$
is the internal $\Hom$-object in $H$-modules.

(2) If $X$ is  $H$, $\HZ$, $\HW$, or $\wHZ$, there is an isomorphism
\[
\pi_\star([[X, H], H]_H)
\cong
\Hom_{H^\star}(H^\star(X), H^\star),
\]
and if  $X$ is $\ul{K}_{*}^M/2$ or $\ul{K}_{*}^W$, there is an isomorphism
\[
\pi_\star([[X, H], H]_H)
\cong
\Ext^{1}_{H^\star}(H^\star(X), \Sigma^{1,0}H^\star).
\]

(3) If $X$ is $H$, $\HZ$, $\HW$, or $\wHZ$ we have an isomorphism
\[
\Ext^{s}_{\A^\star}(H^\star(X), H^\star)
\to
\Ext^{s}_{\A_\star}(H_\star, \Hom_{H^\star}(H^\star(X), H^\star)).
\]
If $X$ is $\ul{K}_{*}^M/2$ or $\ul{K}_{*}^W$, there is an isomorphism
\[
\Ext^{s+1}_{\A^\star}(H^\star(X), H^\star)
\to
\Ext^{s}_{\A_\star}(H_\star, \Ext^{1}_{H^\star}(H^\star(X), H^\star)).
\]

The isomorphisms in (3) are induced by the dualization map
\[
(f : M \to H^\star)
\mapsto
(f^\dual : \Hom_{H^\star}(H^\star, H^\star) \to \Hom_{H^\star}(H^\star, M)),
\]
and are described in detail in the proof.
\end{lemma}
\begin{proof}
(1) For $H$ this is true by \cite[Lemma 5.2]{HKO:steenrod}. %
If $X$  is  $\ul{K}_{*}^M/2$, $\HW$, or $\wHZ$
we have the cofiber sequences \eqref{eq:C5}, \eqref{eq:C4} and \eqref{eq:C1} where \eqref{eq:eval} is an equivalence for $2$ out of $3$ of the spectra,
hence \eqref{eq:eval} is also an equivalence for the third spectrum.
When we smash \eqref{eq:C2} and \eqref{eq:C5} with $H$ we get split cofiber sequences, which implies that \eqref{eq:eval} is an equivalence for $\HZ$ and $\ul{K}_{*}^W$.

(2) This follows from the cofiber sequences in \Cref{lem:cofibs},
or by using the spectral sequence in \cite[Theorem IV.4.1]{EKMM}
which collapses on the $E_2$-page. The $E_2$-page is analyzed using \Cref{lem:sess}. %

(3)
This is true for $H$ by \cite[Lemma 5.2]{HKO:steenrod}.
Note that $\Ext_{\A_\star}(H_\star, \A_\star)$ is readily computed with the cobar complex \cite[A1.2.12]{Ravenel}.

Given $0 \to A^{-1} \to A^0 \to A^{1} \to 0$ an exact sequence of $\A^\star$-modules,
such that the canonical map $\RHom_{H^\star}(\A^\star, H^\star) \tensor_{H^\star} \RHom_{H^\star}(A^i, H^\star)\to \RHom_{H^\star}(\A^\star\tensor_{H^\star} A^i, H^\star)$ is an isomorphism,
choose projective $\A^\star$-resolutions such that the following diagram commutes
\[
\begin{tikzcd}
0\ar[r] & A^{-1} \ar[r] & A^{0} \ar[r] & A^{1} \ar[r] & 0 \\
0\ar[r] & P^{-1}_\bullet \ar[r]\ar[u] & P^{0}_\bullet \ar[r]\ar[u] & P^{-1}_\bullet \ar[r]\ar[u] & 0.
\end{tikzcd}
\]
Dualize the diagram with respect to $(-)^\dual = \Hom_{H^\star}(-, H^\star)$
and choose injective $\A_\star$-comodule resolutions $I_\bullet^i$ of $(P_\bullet^i)^\dual$ such that we have a map of distinguished triangles in $\mathcal D(\A_\star)$
\begin{equation} %
\label{eq:proj-inj}
\begin{tikzcd}
(P^{1}_\bullet)^\dual \ar[r]\ar[d] & (P^{0}_\bullet)^\dual \ar[r]\ar[d] & (P^{-1}_\bullet)^\dual \ar[r]\ar[d] &  (P^{1}_\bullet)^\dual[1]\ar[d] \\
I^{1}_\bullet \ar[r] & I^{0}_\bullet \ar[r] & I^{-1}_\bullet \ar[r] & I^{1}_\bullet[1].
\end{tikzcd}
\end{equation}
We then define the map from
$\Ext^{s}_{\A^\star}(A^i, H^\star)
= H^{s}(\RHom_{\A^\star}(P^i_\bullet, H^\star))$
to
$H^{s}(\RHom_{\A_\star}(H_\star, I^i_\bullet))$
by sending
$P^i_\bullet \to H^\star[s]$
to the composite
$H_\star[s] = (H^\star[s])^\dual \to (P^i_\bullet)^\dual \to I^i_\bullet$.
Since \eqref{eq:proj-inj} is commutative, these maps induce a map between the long exact sequences associated to $\Ext_{\A^\star}$ and $\Ext_{\A_\star}$.
Then (3) follows from the 5-lemma applied to the short exact sequences we get from \Cref{lem:sess}. Indeed, for all the spectra either $I^{i}_{\bullet}$ is an injective resolution of $\Hom_{H^\star}(A^i, H^\star)$
or $I^{i}_\bullet[1]$ is an injective resolution of $\Ext^1_{H^\star}(A^i, H^\star)$.
Hence, $H^{s}(\RHom_{\A^\star}(H_\star, I^i_\bullet))$ is
$\Ext^{s}(H_\star, \Hom_{H^\star}(A^i, H^\star))$
or 
$\Ext^{s-1}(H_\star, \Ext^1_{H^\star}(A^i, H^\star))$.
\end{proof}

Finally we arrive at the main result of this section.
In \Cref{sec:2} this isomorphism is extended to all fields of characteristic not 2 of finite virtual cohomological dimension.
\begin{proposition}
\label{prop:S-wHZ}
Over $\R$ the canonical map $\S \to \wHZ$ induces an isomorphism
\[
\Ext^{s,(t,w)}_{\A_\star^\R}(H_\star^\R, H_\star^\R) \cong \Ext^{s,(t,w)}_{\A_\star^\R}(H_\star^\R, H_\star^\R(\wHZ))
\]
for $w, t - s$ fixed and $s \gg 0$.
\end{proposition}
\begin{proof}
By \Cref{lem:GI} and the vanishing in \Cref{cor:Ext-wHZ} it suffices to prove that we have an isomorphism
for $w, t - s$ fixed, $t - s = 0$, or $t - s - w = 0$.
In this case, we make a computation in the cobar complex of $\Ext_{\A_\star^\R}(H_\star^\R, H_\star^\R)$.

  When $t - s - w = 0$: For degree reasons the only elements in the cobar
  complex which lie in Milnor-Witt degree $0 = t - s - w$ are
  generated by $\rho[]$, $[\tau_0]$ and $[\xi_1]$. The cobar
  differential vanishes on all of these generators. The only entering
  differentials are generated by $d(\tau) = \rho\tau_0$,
  $d([\tau_1]) = [\xi_1\vert \tau_0]$, and
  $d([\tau_0\xi_1])=[\tau_0\vert \xi_1] + [\xi_1 \vert \tau_0]$.
  Indeed, for degree reasons, elements in the cobar complex which lie
  in Milnor-Witt degree $1 = t - s - w$ are generated as a module over
  Milnor-Witt degree 0 by the elements $\tau[]$, $[\tau_1]$,
  $[\tau_0\xi_1]$, and $[\xi_1^2]$ (note that
  $[\tau_0^2] = \tau[\xi_1] + \rho[\tau_0\xi_1] +\rho[\tau_1]$).

  When $t - s = 0$: Choose a representative $x$ in the cobar complex
  for an element of $\Ext_{\A_\star^\R}^{s,(t,w)}(H_\star^\R, H_\star^\R)$.  That is, $x$ is
  some sum of monomials in $\rho$, $\tau$, $\xi_i$ and $\tau_i$ in the cobar complex.
  If $x$ is not of the form $\tau^{-w}[\tau_0\vert \dots \vert
  \tau_0]$, then for degree reason $x = \rho^i x'$ for some $x'$ and $i \gg 0$ when $s \gg 0$.
  Since there is no $\rho$-torsion in the cobar complex and the cobar differential is $\rho$-linear,
  $d(x') = 0$.  However,  for $s \gg 0$, every element $x' \in
  \Ext_{\A_\star^\R}^{s',(t',w')}(H_\star^\R, H_\star^\R)$ is $\rho^{i_0}$-torsion in
  $\Ext_{\A_\star^\R}(H_\star^\R, H_\star^\R)$ for some $i_0$ when $s \gg 0$.  Indeed, this is the statement of
  \Cref{lem:GI}.  Hence any potentially nonzero element of
  $\Ext^{s,(t,w)}_{\A_\star^\R}(H_\star^\R, H_\star^\R)$ is of the form $\tau^{-w}[\tau_0\vert
  \dots \vert \tau_0]$.  The differential of such an element is
  $\rho\tau^{-w-1}[\tau_0\vert\tau_0\vert \dots \vert \tau_0]$ if $-w$
  is odd, and $0$ otherwise.

  It remains to check that $\tau^{-2w}[\tau_0]^i$ and $\rho^j[\xi_1]^k$
  map to the generators of $\Ext_{\A_\star^\R}(H_\star^\R, H_\star^\R(\cHZ))$.
  We check this for the cohomological $\Ext$-group.
  Since the map is an $\Ext_{\A^\star_\R}(H^\star_\R, H^\star_\R)$-module map,
  it suffices to check that the generators $[\tau_0]$ and $[\xi_1]$
  map to the generators in $\Ext_{\A^\star_\R}(H^\star_\R(\cHZ), H^\star_\R)$.
  This is done in \Cref{lem:generators-ok}.
  The $\Ext_{\A^\star_\R}(H^\star_\R, H^\star_\R)$-module structure of $\Ext_{\A^\star_\R}(H^\star_\R(\HW), H^\star_\R)$ is described in \Cref{lem:modtau-ext}.
  The $\Ext_{\A^\star_\R}(H^\star_\R, H^\star_\R)$-module structure of $\Ext_{\A^\star_\R}(H^\star_\R(\HZ), H^\star_\R)$ is just the same as in topology.
  This can be seen by either mimicking the proof of \Cref{lem:modtau-ext}
  or by taking complex realization.
\end{proof}

\begin{lemma}
\label{lem:generators-ok}
The map $\Ext_{\A^\star}(H^\star(\S), H^\star) \to \Ext_{\A^\star}(H^\star(\HZ), H^\star)$ maps $h_0$ (represented by $[\tau_0]$ in the cobar complex) in tridegree $(1, 1, 0)$ to the generator of $\Ext^{1,(1,0)}_{\A^\star}(H^\star(\HZ), H^\star)$.

The map $\Ext_{\A^\star}(H^\star(\S), H^\star) \to \Ext_{\A^\star}(H^\star(\HW), H^\star)$ maps $h_1$ (represented by $[\xi_1]$ in the cobar complex) in tridegree $(1, 2, 1)$ to the generator of $\Ext^{1,(2,1)}_{\A^\star}(H^\star(\HW), H^\star)$.
\end{lemma}
\begin{proof}
The map $H^\star(\HZ) \to H^\star(\S)$ combined with \eqref{eq:ses5} gives us a map of short exact sequences
\begin{equation}
\label{eq:S-to-H-res}
\begin{tikzcd}
0 \ar[r] & \A^\star\Sq^1 \ar[r]\ar[d] & \A^\star \ar[r]\ar[d, "="] & \A^\star/\Sq^1 \ar[r]\ar[d] & 0 \\
0 \ar[r] & \A^\star_{>0} \ar[r] & \A^\star \ar[r] & H^\star \ar[r] & 0.
\end{tikzcd}
\end{equation}
Here $\A^\star_{>0}$ is the kernel of $\A^\star \to H^\star$.
Applying $\Hom_{\A^\star}(-, H^\star)$ to \eqref{eq:S-to-H-res}
the map $(\Sq^1 \mapsto 1) \in \Hom_{\A^\star}(\A^\star_{>0}, H^\star)$
is mapped to $(\Sq^1 \mapsto 1) \in \Hom_{\A^\star}(\A^\star\Sq^1, H^\star)$.
Hence, we get that $h_0 \in \Ext^{1,(1,0)}_{\A^\star}(H^\star, H^\star)$
is mapped to the generator of $\Ext^{1,(1,0)}_{\A^\star}(H^\star(\HZ), H^\star)$.

By \eqref{eq:ses3} we have the short exact sequence
\[
0 \to \Sigma^{0,-1}\A^\star(\tau, \Sq^2 + \rho\Sq^1) \to \Sigma^{0,-1}\A^\star \to \Sigma^{0,-1}\A^\star/(\tau, \Sq^2 + \rho\Sq^1) \to 0.
\]
The map $H^\star(\HW) \to H^\star(\S)$ induces the map of short exact sequences
\begin{equation}
\label{eq:S-to-HW-res}
\begin{tikzcd}[column sep=13ex]
K \ar[r, rightarrowtail]\ar[d] & \A^\star \oplus \Sigma^{2,0}\A^\star \ar[r, "{(\cdot\tau, \cdot (\Sq^2 + \rho\Sq^1))}", two heads]
\ar[d, "{(1, 0)}"]
& \Sigma^{0,-1}\A^\star(\tau, \Sq^2 + \rho\Sq^1) \ar[d] \\
\A^\star_{>0} \ar[r, rightarrowtail] & \A^\star \ar[r, two heads] & H^\star .
\end{tikzcd}
\end{equation}
Here $K$ is just the kernel of ${(\cdot\tau, \cdot (\Sq^2 + \rho\Sq^1))}$, identified in \eqref{eq:tom}.
The map $\Sigma^{0,-1}\A^\star(\tau, \Sq^2 + \rho\Sq^1) \to H^\star$ maps $\tau$ to $1$, since $H^\star(H) \to H^\star(\HW) \to H^\star(\S)$ maps $1$ to $1$.
This implies that the middle vertical map is the identity plus zero.
The element $(\Sq^2 \mapsto 1) \in \Hom_{\A^\star}(\A^\star_{>0}, H^\star)$
corresponds to $h_1$.
By \eqref{eq:S-to-HW-res} this maps to
$((\Sq^2, \tau) \mapsto 1) \in \Hom_{\A^\star}(K, H^\star)$
(note that $(\Sq^2, \tau)$ is an element of $K$ by \eqref{eq:Sq2tau}).
Hence $h_1 \in \Ext^{1,(2,1)}_{\A^\star}(H^\star(\S), H^\star)$
maps to the generator of $\Ext^{1,(2,1)}_{\A^\star}(H^\star(\HW), H^\star)$.
\end{proof}

\section{Ext at the prime 2}
\label{sec:2}
In this section we establish an isomorphism of $\Ext$-groups between the motivic sphere spectrum and $\wHZ$ in high filtration over fields of finite virtual cohomological dimension.
We get similar vanishing lines to that of Guillou-Isaksen, but shifted by the virtual cohomological dimension of the base field.
The main observation is \Cref{lem:ext} which also seems to be of possible computational interest.
For instance, combined with \cite[Figure 3]{DI:Real} the lemma recovers Morel's $\pi_1$-conjecture \cite{RSO:April1} up to resolving one $d_2$-differential.

\begin{lemma}
  \label{lem:ext}
Suppose $X$ is $\S$, $\wHZ$, $\HZ$, $H$, $\HW$, or $\HW/2^n$.  Over a
field of characteristic not 2 we have for each $s$ an extension
\begin{align*}
0
\to \Ext^{s}_{\A_\star^\R}(H_\star^\R, H_\star^\R(X)) \tensor_{\Z/2[\rho]} k_*
&\to \Ext^{s}_{\A_\star}(H_\star, H_\star(X)) \\
&\to \Tor_1^{\Z/2[\rho]}(\Ext^{s+1}_{\A_\star^\R}(H_\star^\R, H_\star^\R(X)), k_*)
\to 0.
\end{align*}
\end{lemma}
\begin{proof}
The group $\Ext_{\A_\star}(H_\star, H_\star(X))$ is the cohomology of the cobar complex $C^\bullet(H_\star, H_\star(X))$ given in \eqref{eq:cobar}.
The differentials in the cobar complex are given in terms of $\Delta$, $\eta_L$ and $\eta_R$.
Hence the differentials are all $k_*$-linear (cf.~\eqref{eq:dual}),
and the cobar complex can be identified with
\[
C^\bullet(H_\star, H_\star(X))
\cong
C^\bullet(H_\star^\R, H_\star^\R(X)) \tensor_{\Z/2[\rho]} k_*.
\]
Since the ring $\Z/2[\rho]$ has $\Tor$-dimension 1 the $\Tor$-spectral sequence
\[
\Tor_{i}^{\Z/2[\rho]}(H^{j}(C^\bullet(H_\star^\R, H_\star^\R(X)), k_*)
\implies H^{i+j}(C^\bullet(H_\star, H_\star(X)))
\]
is concentrated along 2-columns and collapses to yield the above extensions.
\end{proof}

\begin{lemma}
\label{lem:vcd}
Let $F$ be a field of characteristic not 2 with $\vcd(F) = n < \infty$.
Then multiplication by $\rho \in k_1(F)$
\[
\rho : k_m(F) \to k_{m+1}(F),
\]
is an isomorphism for $m \geq n$.
\end{lemma}
\begin{proof}[First proof]
Let $I$ be the fundamental ideal in the Witt-ring $W(F)$.
Let $m \geq n$, then $2 I^m = I^{m+1}$, and $I^{m}$ is torsion free \cite[Corollary 35.27]{EKM}, \cite{AP}.
By definition $\langle 1,1 \rangle = 2$.
By the Milnor-conjecture $I^i/I^{i+1} \cong k_{i}(F)$ for all $i$,
and $\rho \in k_1(F)$ has a lift to $\langle 1,1 \rangle \in I^1$.
Hence, the map $\langle1,1\rangle : I^m/I^{m+1} \to I^{m+1}/I^{m+2}$ is an isomorphism,
i.e., $\rho : k_m(F) \to k_{m+1}(F)$ is an isomorphism.
\end{proof}
\begin{proof}[Second proof]
Use the Gysin-sequence for the quadratic extension $F \to F[\sqrt{-1}]$,
cf.~\cite[Proposition 2]{HKO:adams}.
\end{proof}

By \Cref{lem:vcd} multiplication by $\rho$ is an isomorphism in degrees greater than or equal to $n$.
Define
\[
C = \coker((\Z/2[\rho] \tensor_{\Z/2} k_n) \to k_*),
\]
where $k_n$ is considered as a graded $\Z/2$-module concentrated in degree $n$.
Observe that the module $C$ has a free $\Z/2[\rho]$-resolution 
\begin{equation}
  \label{eq:C-resolution}
0 \to \Z/2[\rho]\{J\} \to \Z/2[\rho]\{I\} \to C \to 0
\end{equation}
where $I$ is a set of generators with degrees in $[0,n]$ and $J$ is a
set of generators in degrees $[0,n+1]$, relying on the fact that
$C_k = 0$ for $k > n$.\footnote{The set $I$ can evidently be chosen to
  satisfy the condition. Now say $Y$ is an element of
  $\Z/2[\rho]\{J\}$ of degree $N$ that is at least $n+2$; we will show
  $Y$ decomposes as a linear combination of generators with degrees in
  $[0,n+1]$.  The class $Y$ maps to $\rho^2 \sum_{i\in I}a_i(\rho)x_i$
  by our assumption that the elements of $I$ are concentrated in
  degrees $[0,n]$. But $\rho \sum_{i\in I}a_i(\rho)x_i$ also maps to
  $0$ in $C_{N-1}$ as $N-1 \geq n + 1$, and so there is a linear
  combination $\sum_{j\in J} b_j(\rho)y_j$ mapping to
  $\rho \sum_{i\in I}a_i(\rho)x_i$ by exactness. However, both $Y$ and
  $\rho \sum_{j\in J} b_j(\rho)y_j$ map to the same class, hence are
  equal by exactness. Therefore $Y$ is not a generator of
  $\Z/2[\rho]\{J\}$.}%
\begin{theorem}
\label{thm:2-iso}
For a field $F$ of characteristic not 2 with $\vcd(F) = n < \infty$
the unit $\S \to \wHZ$ induces an isomorphism
\[
\Ext^{s,(t,w)}_{\A_\star}(H_\star,H_\star(\S)) \xrightarrow{\cong}
\Ext^{s,(t,w)}_{\A_\star}(H_\star,H_\star(\wHZ))
\]
for $t-s, w$ fixed and $s \gg 0$.
\end{theorem}
\begin{proof}
By \Cref{lem:ext} it suffices to show that 
\[
\Ext^{s}_{\A_\star^\R}(H_\star^\R, H_\star^\R(\S)) \tensor_{\Z/2[\rho]} k_*
\to 
\Ext^{s}_{\A_\star^\R}(H_\star^\R, H_\star^\R(\cHZ)) \tensor_{\Z/2[\rho]} k_*
\]
and
\[
\Tor_1^{\Z/2[\rho]}(\Ext^{s+1}_{\A_\star^\R}(H_\star^\R, H_\star^\R(\S)), k_*)
\to 
\Tor_1^{\Z/2[\rho]}(\Ext^{s+1}_{\A_\star^\R}(H_\star^\R, H_\star^\R(\wHZ)), k_*)
= 0
\]
are isomorphisms, for $t-s,w$ fixed and $s \gg 0$.
The long exact sequence of $\Tor$-groups associated to the short
exact sequence $0 \to \Z/2[\rho]\otimes k_n \to k_* \to C \to 0$ yields the
following exact sequence. %
\begin{align}
\label{eq:torles}
0
\to
&\Tor_1^{\Z/2[\rho]}(\Ext^{s}_{\A_\star^\R}(H_\star^\R, H_\star^\R(X)), k_*)
\to
\Tor_1^{\Z/2[\rho]}(\Ext^{s}_{\A_\star^\R}(H_\star^\R, H_\star^\R(X)), C)
\to \\
&\Ext^{s}_{\A_\star^\R}(H_\star^\R, H_\star^\R(X)) \tensor_{\Z/2} k_n
\to
\Ext^{s}_{\A_\star^\R}(H_\star^\R, H_\star^\R(X)) \tensor_{\Z/2[\rho]} k_*
\to \nonumber\\
&\Ext^{s}_{\A_\star^\R}(H_\star^\R, H_\star^\R(X)) \tensor_{\Z/2[\rho]} C
\to
0 \nonumber
\end{align}
Hence it suffices to show that
\begin{equation}
\label{eq:TorC}
\Tor_1^{\Z/2[\rho]}(\Ext^{s}_{\A_\star^\R}(H_\star^\R, H_\star^\R(\S)), C)^{(t,w)} = 0
\end{equation}
and
\begin{equation}
\label{eq:TensorC}
(\Ext^{s}_{\A_\star^\R}(H_\star^\R, H_\star^\R(\S)) \tensor_{\Z/2[\rho]} C)^{(t,w)}
\to
(\Ext^{s}_{\A_\star^\R}(H_\star^\R, H_\star^\R(\wHZ)) \tensor_{\Z/2[\rho]} C)^{(t,w)}
\end{equation}
is an isomorphism for $s \gg 0$.

Using the resolution \eqref{eq:C-resolution} we see that \eqref{eq:TorC} is a submodule of $(\Ext^{s+1}_{\A_\star^\R}(H_\star^\R, H_\star^\R) \{ J \})_{t,w}$.
As the set of bidegrees of the indices in $J$ is finite \Cref{prop:S-wHZ} implies that \eqref{eq:TorC} is 0 for $t-s,w$ fixed and $s \gg 0$.
Similarly, the source of \eqref{eq:TensorC} is a quotient of $(\Ext^{s}_{\A_\star^\R}(H_\star^\R, H_\star^\R) \{ I \})_{t,w}$. As the set of bidegrees of the indices in $I$ is finite \Cref{prop:S-wHZ} implies that \eqref{eq:TensorC} is an isomorphism for $t-s,w$ fixed and $s \gg 0$.
\end{proof}

\begin{corollary}
Let $F$ be a field of characteristic not 2 with $\vcd(F)=n$.  Then
$\Ext_{\A_{\star}}^{s,(t,w)}(H_\star, H_\star)$ vanishes if the
following three conditions are satisfied: (1) $s > \frac{1}{2}(m+3)$;
(2) $t-s < s + n + 1$; (3) $t-s > 0$ or $t-s < -n$.
\end{corollary}

\begin{proof}
This follows directly by analyzing the two ends in the short exact
sequence in \Cref{lem:ext} at the motivic sphere spectrum from the vanishing
regions obtained by Guillou and Isaksen over $\R$ described in \Cref{lem:GI}.
\end{proof}

\section{Ext at odd primes}
\label{sec:odd}
In this section we prove that $\S \to \wHZ$ induces an isomorphism of $\Ext$-groups in high filtration at odd primes as well.
Actually, since $H^\star(\wHZ) = H^\star(\HZ)$ by \Cref{lem:ell-KW} we establish that $\S \to \HZ$ induces an isomorphism of $\Ext$-groups in high filtration.
We use a Bockstein spectral sequence to extend Adams's isomorphism \eqref{eq:adams-iso} to general fields.
See \cite[2.2]{Hill} or \cite[p.~51]{Wilson:thesis} for Bockstein
spectral sequences at the prime $2$ or at odd primes over finite
fields.  We introduce a generalization of this for odd primes and
arbitrary fields.

Let $\beta$ denote the Bockstein homomorphism
$\beta : H_{t,w} \to H_{t-1,w}$, and write $B$ for the image of
$\beta$, that is, $B = \im(\beta) \subset H_\star$. Recall that $\beta$
is a derivation and $\beta^2 = 0$. Note that the image $B$ is closed
under multiplication, that is, it is a subalgebra of $H_\star$ without
unit.  Furthermore $B$ is central in $\A_\star$, and the cobar
differential is linear with respect to $B$.

We consider the decreasing filtrations of $H_\star$ and $\A_\star$ induced by the powers of $B$.
The filtration at level $n$ is $F^nH_\star = B^n H_\star$,
and $F^n\A_\star = B^n H_\star$.
This induces a filtration $\{F^nC^\bullet(H_\star, H_\star)\}_{n\geq0}$
of the cobar complex that is given in level $n$ by 
\[
F^nC^s(H_\star, H_\star) = \sum_{i_0 + \dots + i_{s+1} = n} F^{i_0} H_\star \tensor F^{i_1}\oA_\star \tensor \dots \tensor F^{i_s}\oA_\star \tensor F^{i_{s+1}} H_\star.
\]
Since the cobar differential is linear with respect to $B$ this is in fact a filtration of complexes, and of algebras,
hence a filtration of differential graded algebras.

\begin{lemma}
The filtration quotients of the generalized Bockstein filtration are
\[
\frac{F^nC^\bullet(H_\star, H_\star)}
{F^{n+1}C^\bullet(H_\star, H_\star)}
\cong
\frac{B^nH_\star}{B^{n+1}H_\star} \tensor_{\Z/\ell} C^\bullet_\Top.
\]
The tensor product $\tensor_{\Z/\ell}$ is the tensor product of graded $\Z/\ell$-modules.
\end{lemma}
\begin{proof}
Since the mod $\ell$ topological and motivic dual Steenrod algebras are the same up to the coefficients $H_\star$ (cf.~\cite{Milnor:Steenrod}, \cite[Theorem 12.6]{Voevodsky:power}, \cite[Theorem 5.6]{HKO:steenrod}), the statement is clear considered as an isomorphism of graded abelian groups.
We have to show that the differentials are correct:
Consider a general element $x = \sum_{I}a_I\gamma_I \in C^\bullet(H_\star, H_\star)$,
where $a_I \in H_\star$ and $\gamma_I$ is a bar of monomials in $\xi_i$'s and $\tau_j$'s.
The differential of $x$ is given by \eqref{eq:cobar-diff}
\begin{equation}
\label{eq:xdiff}
d(x) = \sum_I (a_I d(\gamma_I) + \beta(a_I)[\tau_0 \vert \gamma_I]).
\end{equation}
If $a_I \in B^n H_\star$,
then $\beta(a_I) \in B^{n+1} H_\star$.
Hence, modulo higher filtration \eqref{eq:xdiff} is the topological differential linear with respect to $H_\star$.
\end{proof}

\begin{corollary}
\label{cor:general-bockstein-ss}
There is a strongly convergent spectral sequence
\[
E_1^{s,t,w,n} = \left(\frac{B^nH_\star}{B^{n+1}H_\star} \tensor_{\Z/\ell} \Ext^{*,(\star)}_{\Top}\right)_{s,(t,w)}
\implies \Ext^{s,(t,w)}_{\A_\star}(H_\star, H_\star),
\]
with $d_r$-differential
$d_r : E_r^{s,t,w,n} \to E_r^{s+1,t,w,n+r}$.
\end{corollary}
\begin{proof}
Given a differential graded algebra $A$ with a decreasing filtration $F^j A$
there is a spectral sequence
\[
H^{i}(F^{j}A/F^{j+1}A) \implies H^{i}(A),
\]
cf.~\cite[5.4.8]{Weibel}.
The $d_1$-differential is induced by the differential of $A$.

In a fixed tridegree $(s, t, w)$ the filtration is finite.
Indeed, $B_{t,w} = 0$ for $t > -1$ or $t < w$,
and both $H_{t,w}$ and $\A_{t,w}$ are 0 for $t < 2w$.
So $(F^n C^\bullet(H_\star, H_\star))^{(t, w)} = 0$ for $n > t - 2w$.
\end{proof}

\begin{theorem}
\label{thm:ell-iso}
For $\ell$ an odd prime and $F$ a field of characteristic not $\ell$ the unit $\S \to \HZ$ induces an isomorphism
\[
\Ext^{s,(t,w)}_{\A_\star}(H_\star, H_\star) \cong \Ext^{s,(t,w)}_{\A_\star}(H_\star, H_\star(\HZ))
\]
for $t - s, w$ fixed and $s \gg 0$.
\end{theorem}
\begin{proof}
Let $t - s = k$.
Consider the map of spectral sequences $E_1^{s,t,w,n}(\S)\to E_1^{s,t,w,n}(\HZ)$
we get from \Cref{cor:general-bockstein-ss}.
This map is induced by $f: \Ext^{*,(\star)}_{\Top}(\S) \to \Ext^{*,(\star)}_{\Top}(\HZ)$.
The part of the $E_1$-page converging to $\Ext^{s,(t, w)}_{\A_\star}(H_\star, H_\star)$ is of the form
\[
\left\{\bigoplus_{(s,t,w) = (s',t',w') + (0, t'', w'')} \left(\frac{B^nH_\star}{B^{n+1}H_\star}\right)_{t'',w''} \tensor \Ext^{s',(t',w')}_{\Top} \right\}_n.
\]
For this to be nonzero:
We need $t'' \leq 0$, so $t - t'' = t' \geq t$.
We need $t' \geq 2w' = 2(w - w'') \geq 2(w - t'') = 2(w - t + t')$,
since $t'' \geq w''$.
Hence,
$2(t - w) \geq t'$, or $t - 2w \geq t' - t = t' - s' - k$.
But, $(2\ell - 3)s > t - 2w + k$ for $s \gg 0$,
hence $(2\ell - 3)s' > t' - s' - k$,
so $f$ is an isomorphism by \eqref{eq:adams-iso} for $w, t - s = k$ fixed and $s \gg 0$.
That is $E_1^{s,t,w,n}(\S) \cong E_1^{s,t,w,n}(\HZ)$ for all $n$, so we conclude
\[
\Ext^{s,(t,w)}_{\A_\star}(H_\star, H_\star) \cong \Ext^{s,(t,w)}_{\A_\star}(H_\star, H_\star(\HZ)).
\]
\end{proof}

\section{Strong convergence for the sphere spectrum}
\label{sec:strong}
In this section we establish strong convergence of the motivic Adams spectral sequence for the sphere spectrum over general fields in positive stems, and everywhere for $\S/\ell^n$ over fields of finite virtual cohomological dimension if $\ell = 2$.
The Adams spectral sequence is always conditionally convergent to the homotopy groups of the $H$-completion. We check that the derived $E_\infty$-term vanishes in a range, and hence that the spectral sequence converges strongly to the abutment in this range. As discussed in \Cref{sec:recollection}, the $H$-completion can be identified with the $(\ell,\eta)$-completion over perfect fields of characteristic not $\ell$ \cite{Mantovani}.

As corollaries of \Cref{thm:2-iso} and \Cref{thm:ell-iso} we get:
\begin{corollary}
\label{cor:strong}
The mod $\ell$ motivic Adams spectral sequence for the sphere spectrum is strongly convergent in tridegree $(s,t,w)$ for $t-s > 0$ over fields $F$ with finite virtual cohomological dimension if $\ell = 2$.
\end{corollary}
\begin{proof}
If $\Ext^{s,(t,w)}_{\A_\star}(H_\star, H_\star(\wHZ)) = 0$
for $t-s-1,w$ fixed and $s \gg 0$
then there can only be finitely many nonzero differentials exiting
$E_r^{s,t,w}$ by \Cref{thm:2-iso} and \Cref{thm:ell-iso}.
By \Cref{cor:Ext-wHZ} and \Cref{lem:Bockstein-Adams} we have the vanishing
$\Ext^{s,(t,w)}_{\A_\star}(H_\star, H_\star(\wHZ)) = 0$ for $t - s > 0, w$ fixed and $s \gg 0$.
Hence we have strong convergence in tridegrees $(s,t,w)$ with $t - s > 1$.
When $t - s = 1$, the differentials enter the column with $t - s = 0$.
But any nonzero element $y \in \Ext^{s,(t,w)}_{\A_\star}(H_\star, H_\star(\wHZ))$ with $t - s = 0$ and $s \gg 0$ is a multiple of $h_0$, and furthermore, $h_0^ky \neq 0$ for all $k$. For an element $x \in E_r^{s,t,w}, t-s=1$ with differential $d_r(x)=y$, it follows that $0 = d_r(h_0^k x) = h_0^k d_r(x) = h_0^ky$. But $h_0^ky = 0$ only when $y=0$. 
\end{proof}

Next we show strong convergence of the motivic Adams spectral sequence for $\S/\ell^n$. This requires us to compare with the motivic Adams spectral sequence for $\cHZ/\ell^n$, which splits into the spectral sequences for $\HZ/\ell^n$ and $\wHZ/\ell^n$ in high filtration.
This is done in the lemmas below.

\begin{lemma}
\label{cor:Ext-l}
For $\ell$ a prime the unit map
$\S/\ell^n \to \wHZ/\ell^n$ induces an isomorphism of $\Ext$-groups
\[
\Ext^{s,(t,w)}_{\A_\star}(H_\star, H_\star(\S/\ell^n)) \cong \Ext^{s,(t,w)}_{\A_\star}(H_\star, H_\star(\wHZ/\ell^n))
\]
for $t - s, w$ fixed and $s \gg 0$.
If $\ell$ is odd the latter group is $\Ext^{s,(t,w)}_{\A_\star}(H_\star, H_\star(\HZ/\ell^n))$.
\end{lemma}
\begin{proof}
The map of cofiber sequences from $\S \to \S \to \S/\ell^n$ to $\wHZ \to \wHZ \to \wHZ/\ell^n$ induces a map of long exact sequences of $\Ext$-groups.
For $s \gg 0$ two thirds of the maps are isomorphisms, so the 5-lemma implies that all the maps are isomorphisms.

To get the final claim when $\ell$ is odd use the cofiber sequence \eqref{eq:C1}.
\end{proof}

\begin{lemma}
\label{lem:Adams-ell-n}
Let $\ell$ be a prime and $F$ a perfect field of finite virtual cohomological dimension if $\ell = 2$.
Then in the motivic Adams spectral sequence for $\wHZ/\ell^n$
we have $E_\infty^{s,t,w}(\wHZ/\ell^n) = 0$
for $s > t - s + n + \vcd(F)$.
\end{lemma}
\begin{proof}
For $\ell$ an odd prime we have
$E_r^{s,t,w}(\wHZ/\ell^n) \cong E_r^{s,t,w}(\HZ/\ell^n)$,
for all $r \geq 2$ by \Cref{lem:ell-KW}.
For $\ell = 2$ we have
$E_r^{s,t,w}(\wHZ/2^n) \cong E_r^{s-1,t-1,w}(\HZ/2^n)\oplus E_r^{s,t,w}(\HW/2^n)$,
for $s > 0$. Indeed by \Cref{cor:Ext-wHZ}, this is true for $r = 2$.
Furthermore, any element of $E_r^{s,t,w}(\HZ/2^n)$ %
is a $h_0$-multiple,
while any element of $E_r^{s,t,w}(\wHZ/2^n), s > 1$, is a multiple of $h_1$.
Because of the $h_0$- and $h_1$-linearity of the differentials there cannot be any differential between $E_r^{s-1,t-1,w}(\HZ/2^n)$ and $E_r^{s,t,w}(\HW/2^n)$.
Hence it suffices to inspect the spectral sequences $E_r^{s-1,t-1,w}(\HZ/2^n)$ and $E_r^{s,t,w}(\HW/2^n)$ independently.
By a variant of \Cref{lem:Bockstein-Adams} applied to $\HZ/\ell^n$ we get that $E_n^{s,\star}(\HZ/\ell^n)= 0$ for $s > n$.
By \Cref{lem:vanish-HW}, we have $E_n^{s,t,w}(\HW/2^n)= 0$ for $s > t - s + n + \vcd(F)$.
\end{proof}

\begin{lemma}
\label{lem:vanish-HW}
Let $F$ be a perfect field of finite virtual cohomological dimension.
Then in the motivic Adams spectral sequence for $\HW/2^n$
we have
\begin{equation}
\label{eq:vanish-En-HW}
E_{n+1}^{s,t,w}(\HW/2^n) = 0
\end{equation}
for $s > t - s + n + \vcd(F)$.
\end{lemma}
\begin{proof}
We have the short exact sequence
\begin{equation}
\label{eq:HW-ell-split}
0 \to H_\star(\HW) \to H_\star(\HW/2^n) \to  \Sigma^{1,0}H_\star(\HW) \to 0.
\end{equation}
This induces the following long exact sequence.
\begin{align}
\label{eq:HW/2n-les}
\dots
&\to
\Ext^{s,(t,w)}_{\A_\star}(H_\star, H_\star(\HW))
\to
\Ext^{s,(t,w)}_{\A_\star}(H_\star, H_\star(\HW/2^n))\\
&\to
\Ext^{s,(t-1,w)}_{\A_\star}(H_\star, H_\star(\HW))
\to
\Ext^{s+1,(t,w)}_{\A_\star}(H_\star, H_\star(\HW))
\to \dots \nonumber
\end{align}
We first work over $\R$.
When $n = 1$ the abutment of $E_\infty(\HW/2)$ is $\W(\R) = \Z/2$.
Then in the long exact sequence $1 \in \Ext^{0,(1-1,0)}_{\A_\star^\R}(H_\star^\R, H_\star^\R(\HW))$
must map to
$\rho[\xi_1] \in \Ext^{1,(1,0)}_{\A_\star^\R}(H_\star^\R, H_\star^\R(\HW))$.
Hence,
\[
\Ext^{s+1}_{\A_\star^\R}(H_\star^\R, H_\star^\R(\HW/2))
\cong
\begin{cases}
\Sigma^{2s,s}H_\star^\R/(\tau,\rho) & s > 0 \\
0 & \text{otherwise}.
\end{cases}
\]
Then \Cref{lem:ext} implies \eqref{eq:vanish-En-HW} when $n=1$,
since $F$ has finite virtual cohomological dimension.
When $n > 1$, the boundary maps in \eqref{eq:HW/2n-les} must be zero for the abutment of $E_\infty(\HW/2^n)$ to have the correct size. Hence,
\[
\Ext^{s,(t,w)}_{\A_\star^\R}(H_\star^\R, H_\star^\R(\HW/2^n))
\cong
\Ext^{s,(t,w)}_{\A_\star^\R}(H_\star^\R, H_\star^\R(\HW))
\oplus
\Ext^{s,(t-1,w)}_{\A_\star^\R}(H_\star^\R, H_\star^\R(\HW))
\]
and the extension in \Cref{lem:ext} implies
\[
\Ext^{s,(t,w)}_{\A_\star}(H_\star, H_\star(\HW/2^n))
\cong
\Ext^{s,(t,w)}_{\A_\star}(H_\star, H_\star(\HW))
\oplus
\Ext^{s,(t-1,w)}_{\A_\star}(H_\star, H_\star(\HW))
\]
over any perfect field when $n > 1$.
Over $\R$, when $n > 1$, there has to be a $d_n$ differential from
$1 \in \Ext^{0,(1-1,0)}_{\A_\star}(H_\star, H_\star(\HW))$
to $\rho^n[\xi_1]^n \in \Ext^{n,(n,0)}_{\A_\star}(H_\star, H_\star(\HW))$
for the abutment to have the correct size $\W(\R) = \Z/2^n$.
By base change we also have this differential over $\Q$
(if $n = 2$, we can have $d_n(1) = (\rho^n + u)[\xi_1]^n$ for some $u \in k_2(\Q)$,
$\rho^2 + u \neq 0$, but since $\rho u = 0$ this does not change the conclusion),
hence by base change from $\Q$ we have $d_n(1) = \rho^n[\xi_1]^n$ over any field of characteristic 0.
Hence, $E_{n+1}^{s,t,w} = 0$ for $s > \vcd(F) + n$.

Fields of positive characteristic have finite $2$-cohomological dimension by our assumption of finite virtual cohomological dimension, cf.~\Cref{rmk:fin-fin-vcd}.
Hence $\Ext^{s,(t,w)}_{\A_\star}(H_\star, H_\star(\HW/2^n)) = 0$ for $t-s, w$ fixed and $s > t -s + \cd_2(F)$ is automatic.
\end{proof}

\begin{corollary}
\label{cor:S/elln-conv}
For a prime $\ell$ the motivic Adams spectral sequence for $\S/\ell^n$ is strongly convergent
over fields $F$ with finite virtual cohomological dimension if $\ell = 2$.
\end{corollary}
\begin{proof}
If $F$ is perfect then
$E_2^{s,t,w}(\S/\ell^n) \cong E_2^{s,t,w}(\wHZ/\ell^n)$
for $t-s, w$ fixed and $s > s_0(t-s,w,2)$ by \Cref{cor:Ext-l}.
If $F$ has positive characteristic and is not necessarily perfect,
then $E_2^{s,t,w}(\S/\ell^n) \cong E_2^{s,t,w}(\HZ/\ell^n)$
for $t-s, w$ fixed and $s > s_0(t-s,w,2)$ by the observation at the end of the proof of \Cref{cor:Ext-l}.

Inductively,
$E_r^{s,t,w}(\S/\ell^n) \cong E_r^{s,t,w}(\wHZ/\ell^n)$
for $t-s,w$ fixed and $s > \max\{s_0(t - s-1, w, r-1), s_0(t-s, w, r-1), s_0(t-s+1, w, r-1)+r-1\} = s_0(t-s,w,r)$.
But $E_r^{s,t,w}(\wHZ/\ell^n) = 0$ for $r > n$, $t-s, w$ fixed and $s \gg 0$ by \Cref{lem:Adams-ell-n}.
Hence $E_r^{s,t,w}(\S/\ell^n) = 0$ for $t-s, w$ fixed $s \gg 0$ for some $r$.
Hence the spectral sequence is strongly convergent.
\end{proof}
As mentioned in the introduction, \Cref{cor:S/elln-conv} suggests a general strategy for extending topological computations with the Adams spectral sequence to motivic homotopy theory:
First prove that the computation holds for $\S/\ell^n$ by mimicking the classical argument, and then pass to the limit over $n$.

In Milnor-Witt degree 0 we note the following much easier strong convergence result.
\begin{lemma}
\label{lem:strong-MW0}
The motivic Adams spectral sequence for the sphere spectrum is strongly convergent in Milnor-Witt degree 0 over any field.
\end{lemma}
\begin{proof}
The $E_2$-page of the spectral sequence is zero in negative Milnor-Witt degrees.
Hence there are no exiting differential from Milnor-Witt degree 0,
so $RE_\infty^{MW=0} = 0$.
\end{proof}

\section{The $\ell$-Bockstein spectral sequence}
\label{sec:Bockstein}
In this section we study the $\ell$-Bockstein spectral sequence for motivic cohomology. As we show below this is the same as the mod $\ell$ Adams spectral sequence for $\HZ$. In \Cref{sec:2} and \Cref{sec:odd} we proved that this spectral sequence is isomorphic (outside of Milnor-Witt degree 0 at the prime 2) in high filtration with the motivic Adams spectral sequence for the sphere spectrum.
We investigate strong convergence of the $\ell$-Bockstein spectral sequence,
and show that the $\ell$-Bockstein spectral sequence is not strongly convergent over number fields. Comparing with the motivic sphere spectrum we conclude that the motivic Adams spectral sequence is not strongly convergent over number fields.

For $\ell$ a prime consider the tower
\begin{equation}
\label{eq:Bockstein-tower}
\dots \HZ \xrightarrow{l} \HZ \xrightarrow{l} \HZ.
\end{equation}
The cofiber of each map in this tower is $H = \HZ/\ell$.
After applying $\pi_\star(-)$ we get a spectral sequence conditionally convergent to $\pi_\star(\HZ\ellcomp)$,
whose $E_1$-page is
\[
E^{s,t,w}_1 = H_{t,w}(F;\Z/\ell),\ s \geq 0.
\]
The map of cofiber sequences
\begin{equation}
\label{eq:cofib-Bock-Adams}
\begin{tikzcd}
\S \ar[r, "\ell"]\ar[d] & \S \ar[r]\ar[d] & \S/\ell\ar[d] \\
\Sigma^{-1,0}\ol{H} \ar[r] & \S \ar[r] & H
\end{tikzcd}
\end{equation}
induces a map of spectral sequences
\begin{equation}
\label{eq:Bock-Adams}
E_r^{s,t,w}(\text{Bockstein}) \to E_r^{s,t,w}(\text{Adams}).
\end{equation}
In the next lemma we show that \eqref{eq:Bock-Adams} is an isomorphism.
This is essentially the observation that \eqref{eq:Bockstein-tower} is an Adams resolution.
\begin{lemma}
\label{lem:Bockstein-Adams}
The map \eqref{eq:Bock-Adams} is an isomorphism from the $E_2$-page and onwards.
\end{lemma}
\begin{proof}
The $E_1$-page of the Bockstein spectral sequence takes the form
\[
E^{s,t,w}_1(\text{Bockstein}) = \pi_{t,w}(\HZ/\ell),\ s\geq 0,
\]
and the $d_1$-differential is simply the Bockstein $\beta$.
Hence, the $E_2$-page is
\[
E^{s,t,w}_2(\text{Bockstein}) = \begin{cases}
  \ker(\beta : H_{t,w})/\im(\beta : H_{t+1,w}) & s > 0 \\
  \ker(\beta : H_{t,w}) & s = 0 \\
  0 & s < 0.
\end{cases}
\]
The cohomology of $\HZ$ is %
$
\A_\star \cotensor_{E_\star(0)} H_\star,
$
where $E_\star(0)$ is the Hopf-algebra $(H_\star, H_\star[\tau_0]/(\tau_0^2))$.
The usual base change theorem \cite[A1.13.12]{Ravenel} implies
\[
E_2^{s,t,w}(\text{Adams}) = \Ext_{E_\star(0)}^{s,(t,w)}(H_\star, H_\star).
\]
The latter is readily computed with the cobar complex to be
\begin{equation}
\label{eq:E2-HZ}
E_2^{s,t,w}(\text{Adams}) = \ker(\beta : H_{t,w})h_0^s/\im(\beta : H_{t+1,w})h_0^{s}.
\end{equation}
It only remains to observe that \eqref{eq:Bock-Adams} induces an isomorphism of these groups.
The map is surjective, since already on the $E_1$-page the map sends $x \in H^{p,w} = E^{s,t,w}_1(\text{Bockstein})$ to $x h_0^s \in E^{s,t,w}_1(\text{Adams})$.

Indeed, the map is induced by powers of the map $\S \to \Sigma^{-1,0}\oH$ in \eqref{eq:cofib-Bock-Adams}.
This is the map representing $h_0 = [\tau_0]$ in the cobar complex because of the commutative diagram
\[
\begin{tikzcd}
  && H \wedge \Sigma^{-1,0} H\ar[d] \\
\S \ar[r]\ar[rru, "\tau_0", dashed]\ar[d, "\ell"] & \Sigma^{-1,0}\oH \ar[r]\ar[d] & H \wedge \Sigma^{-1,0}\oH \ar[d] \\
\S \ar[r, "="] & \S \ar[r] & H \wedge \S
\end{tikzcd}
\]
where the vertical column is a cofiber sequence.
The dashed lift exists since the composite $\S \xrightarrow{\ell} \S \to H\wedge \S$ is zero.

Hence \eqref{eq:Bock-Adams} maps $x$ to $x h_0^s$.
So the map is surjective on the $E_2$-page.
Furthermore \eqref{eq:Bock-Adams} is injective, since the $d_1$-differential adds the same relations, and the relations are mapped identically by \eqref{eq:Bock-Adams}. %
\end{proof}

\begin{corollary}
The following are equivalent:
\begin{enumerate}
\item The mod $\ell$ motivic Adams spectral sequence for $\HZ$ is strongly convergent in bidegree $(t, w)$,
\item The $\ell$-Bockstein spectral sequence for motivic cohomology is strongly convergent in bidegree $(t, w)$.
\end{enumerate}
\end{corollary}
\begin{proof}
  This is a consequence of the mapping lemma \cite[Exercise
  5.2.3]{Weibel}.
\end{proof}

\subsection*{Strong convergence of the $\ell$-Bockstein spectral sequence}
Boardman's criteria for strong convergence \cite[7.1]{Boardman} of the $\ell$-Bockstein spectral sequence amounts to showing $\lim^1_r Z^{p,w}_r = 0$.
If we use $\ell$-local motivic cohomology, and write  $\partial : H^{p,w}(F;\Z/\ell)\to H^{p+1,w}(F;\Z_{(\ell)})$ for the relevant connecting homomorphism, we must then show the group
\begin{equation}
\label{eq:boardman}
\lim^1_r Z^{p,w}_r 
= \lim^1_r \partial^{-1}(\ell^{r-1} \cdot H_{p-1,w}(F;\Z_{(\ell)}))
= \lim^1_r \partial^{-1}(\ell^{r-1} \cdot H^{-p+1,-w}(F;\Z_{(\ell)}))
\end{equation}
is trivial.
For all $r$ we have short exact
sequences
\[
0 \to \im (\pr) \to Z_{r+1}^{-p,-w} \to \im(\partial) \cap \ell^r\cdot H^{p+1,w}(F;\Z_{(\ell)}) \to 0,
\]
where $\pr : H^{p,w}(F;\Z_{(\ell)}) \to H^{p,w}(F;\Z/\ell)$.
The last group is identified with $\im(\partial) \cap \ell^r \cdot H^{p+1,w}(F;\Z_{(\ell)}) = {}_\ell (\ell^r\cdot H^{p+1,w}(F;\Z_{(\ell)}))$.
Hence the $\lim$-$\lim^1$ sequence implies
\begin{equation}
\label{eq:RE}
\lim_r Z_{r+1}^{p,w} \cong  \lim_r {}_\ell (\ell^r\cdot H^{p+1,w}(F;\Z_{(\ell)}))
\text{ and } \lim^1_r Z_{r+1}^{p,w} \cong \lim^1_r {}_\ell (\ell^r\cdot H^{p+1,w}(F;\Z_{(\ell)})).
\end{equation}
So we have reduced the vanishing of \eqref{eq:boardman} to a question
on the structure of motivic cohomology of the base field.

\begin{lemma}
\label{lem:lim-assumption}
If $\lim_r \ell^r B = 0$,
then
$\lim_r {}_{\ell}(\ell^r  B) = 0$. If the exponent of $\ell$-torsion for $B$ is bounded, then 
$\lim^1_r {}_{\ell}(\ell^r  B) = 0$.
\end{lemma}
\begin{proof}
Consider the exact sequence
\[
0 \to {}_{\ell} (\ell^r  B) \to \ell^r  B \xrightarrow{\ell} \ell^r  B \to \ell^r  B/\ell^{r+1} B \to 0.
\]
Let $X_r$ be the cokernel $\coker({}_{\ell} (\ell^r B) \to \ell^r
B)$. From the exactness of the above sequence, it follows that $X_r$
is the kernel of the quotient map
$\ell^r B \to \ell^r B/\ell^{r+1} B$; hence there is an isomorphism
$X_r \cong \ell^{r+1} B$.  The $\lim$-$\lim^1$ long exact sequence from the short exact sequence defining $X_r$ is
\[
0 \to \lim_r {}_\ell(\ell^rB) \to 0 \to \lim_r X_r \to \lim^1_r {}_{\ell}(\ell^r B) \to \lim^1_r \ell^r B
\to  \lim^1_r X_r \to 0.
\]
Thus we obtain the vanishing of $\lim_r {}_{\ell}(\ell^r B)$.

Now observe that if the exponent of $\ell$-torsion for $B$ is bounded,
then ${}_\ell (\ell^r B) = 0$ for sufficiently large $r$. Thus the trivial
Mittag-Leffler condition for the tower ${}_\ell (\ell^r B)$ is satisfied
and the vanishing of $\lim^1_r {}_\ell (\ell^r B)$ follows.
\end{proof}

Let $F$ be a number field.
We have the localization sequence
\begin{equation}
\label{eq:num-field-localization}
\dots \to H^{p,w}(\OFS;\Z_{(\ell)}) \to H^{p,w}(F;\Z_{(\ell)})\to \oplus_{\mathfrak p} H^{p-1,w-1}(k({\mathfrak p});\Z_{(\ell)})\to \dots.
\end{equation}
For $\mathcal{S} \supset \{\ell, \infty \}$ a finite set of places, both 
$H^{p,w}(\OFS;\Z_{(\ell)})$ and $H^{p,w}(k({\mathfrak p});\Z_{(\ell)})$ are finite groups for all $(p,w) \neq (0,0)$
\cite[Section 14]{Levine99}.
\begin{lemma}
\label{lem:num-field}
For a number field $F$ and any $(p,w) \neq (0,0)$, we have
\[
\ell^r \cdot H^{p,w}(F;\Z_{(\ell)}) \cong \ell^r \cdot \oplus_{\mathfrak p} H^{p-1,w-1}(k({\mathfrak p});\Z_{(\ell)})
\]
for $r \gg 0$.
\end{lemma}
\begin{proof}
This follows since $H^{p,w}(\OFS;\Z_{(\ell)})$ is finite for $(p,w) \neq (0,0)$
and the localization sequence.
\end{proof}
  
\begin{lemma}
\label{lem:B-nonvanish}
Let $B = \oplus_{j\geq 1} \Z/\ell^{k_j}$, for some increasing sequence of integers $\{k_j\}_{\geq 1}$.
Then $\lim^1_r {}_{\ell}(\ell^r B) \neq 0$.
\end{lemma}
\begin{proof}
Write $B = \oplus_{j\geq 1} \Z/\ell^{k_j}\{e_j\}$,
and consider the element
\[
x \in \prod_r {}_{\ell}(\ell^r B) :
\left(
r \mapsto \begin{cases}
    \ell^{k_j-1}e_j & r = k_j \\
    0 & r \neq k_j
\end{cases}\in {}_{\ell} \ell^{r}B
\right)
.
\]
The image of $x$ in $\lim^1_r {}_{\ell}\ell^r B$ is nonzero.
\end{proof}

\begin{lemma}
\label{lem:B-split}
Let $F$ be a number field with primes $\mathfrak p$ in $\OFS$, $\mathcal S\supset \{\ell, \infty \}$ a finite set of places.
Then there exists a split surjection
\[
\oplus_{\mathfrak p} H^{1,w}(k(\mathfrak p);\Z)
\to
\oplus_{j\geq 1} \Z/\ell^{k_j},
\]
for some increasing sequence of integers $k_1 < k_2 < k_3 < \dots$.
\end{lemma}
\begin{proof}
By Quillen's computation of algebraic $K$-theory of finite fields we have
\[
H^{1,w}(k(\mathfrak p);\Z_{\ell})
\cong \Z_{(\ell)}/({\mathfrak p}^w - 1)
\cong \Z/\ell^{\nu_\ell({\mathfrak p}^w-1)}.
\]
Here $\nu_\ell$ is the $\ell$-adic valuation.
Note that when $q \equiv 1 \bmod \ell$ then
$\nu_\ell(q^w - 1) = \nu_{\ell}(q-1) + \nu_{\ell}(w)$.

Let $k_0 = 0$. We will find primes $\mathfrak p$ in $F$ inductively whose $\ell$-adic valuations are increasing.
By Dirichlet's theorem about primes in arithmetic progressions \cite[Theorem 4.1.2]{Serre:arithmetic}
we can for any integer $k_{j}$ find a prime $p$ with $\ell$-adic valuation $\nu_{\ell}(p) > k_j$. %
Choose a prime $\mathfrak p$ of $F$ lying over $p$ and set $k_{j+1} = \nu_\ell({\mathfrak p}^w - 1)$.
\end{proof}

\begin{corollary}
\label{lem:HZ-notconv}
The motivic Adams spectral sequence for $\HZ$ has nonvanishing derived $E_\infty$-term in tridegrees $(s, t, w)$ with $t - s = -1, w < -1, s \geq 0,$ over any number field.
That is,
$RE_\infty^{s,t,w}(\HZ) \neq 0$ when $ t-s=-1,w < -1, s \geq 0$.
In particular the motivic Adams spectral sequence for $\HZ$ over a number field is \emph{not} strongly convergent.
\end{corollary}
\begin{proof}
By \Cref{lem:num-field} we have
\[
\lim^1_r {}_{\ell}(\ell^r \cdot H^{2,w+1}(F;\Z_{(\ell)}))
\cong  \lim^1_r {}_{\ell}(\ell^r\cdot \oplus_{\mathfrak p} H^{1,w}(k({\mathfrak p});\Z_{(\ell)})).
\]
The right hand side surjects onto
$\lim^1_r {}_{\ell} (\ell^r\cdot \oplus_{j\geq 1} \Z/\ell^{k_j})$ by \Cref{lem:B-split},
which is nonzero by \Cref{lem:B-nonvanish}.
\end{proof}

\begin{corollary}
\label{cor:S-notconv}
The motivic Adams spectral sequence for the sphere spectrum has nonvanishing derived $E_\infty$-term in tridegrees $(s, t, w)$ with $t - s = -1, w < -1, s \gg 0,$ over any number field.
That is,
$RE_\infty^{s,t,w}(\S) \neq 0$ when $t-s=-1,w < -1, s \gg 0$.
In particular the motivic Adams spectral sequence for the sphere spectrum over a number field is \emph{not} strongly convergent.
\end{corollary}
\begin{proof}
The kernel of $E_2^{s,t,w}(\S) \to E_2^{s,t,w}(\HZ)$ is finite for $t-s < 0, w < -1$. Indeed, over a number field the $\Z/2[\rho]$-module $C$ defined in \eqref{eq:C-resolution} is such that $C_3 = 0$. Hence, the $\Tor$-term in \Cref{lem:ext} is zero.
Similarly, the $-\tensor k_*$-term of in \Cref{lem:ext} only sees the part of $k_n(F)$ for $n > 2$, which is finite.
Hence, since we have an isomorphism
$E_2^{s,t,w}(\S) = E_2^{s,t,w}(\HZ)$ for $t - s, w$ fixed and $s \gg 0$ (\Cref{thm:2-iso}, \Cref{thm:ell-iso}), we get that $RE_\infty^{s,t,w}(\S) = RE_\infty^{s,t,w}(\HZ)$ for $t - s = -1, w < -1, s \gg 0$, which is nonzero by \Cref{lem:HZ-notconv}.
\end{proof}

\section{Bounds on exponents of motivic stable stems}
\label{sec:bounds}
As a simple corollary of \Cref{lem:GI}, \Cref{lem:ext} and \Cref{cor:general-bockstein-ss}
we get some very coarse bounds on the exponents of the stable motivic homotopy groups.
This is a partial answer to a problem posed in \cite[p.~2]{ALP}.
\begin{corollary}
\label{cor:tors-2}
Let $F$ be a perfect field of characteristic not 2 with finite virtual cohomological dimension.
If $t > 1$ and $t - w > 0$ the exponent of $\pi_{t,w}(\S_{2,\eta}^\wedge)$ is bounded by $2^{\max\{\lceil(t-w + 1)/2\rceil, t + \vcd(F)\}}$.
\end{corollary}
\begin{proof}
Let $n = \vcd(F)$.
By \Cref{lem:ext}
$\Ext^{s}_{\A_\star}(H_\star, H_\star)$
is an extension of
$\Ext^{s}_{\A_\star^\R}(H_\star^\R, H_\star^\R) \tensor_{\Z/2[\rho]} k_*$
and
$\Tor_1^{\Z/2[\rho]}(\Ext^{s+1}_{\A_\star^\R}(H_\star^\R, H_\star^\R), k_*)$.
In tridegree $(s, t, w)$ and Milnor-Witt degree $m = t - s - w$,
the first group is zero for $s > t - s + 1 + n$ and $s > \frac{1}{2}(m + 3)$ by \Cref{lem:GI}.
The second group is zero for $s > t - s + n$ and $s + 1 > \frac{1}{2}(m+2)$.
Hence, $\Ext^{s,(t,w)}_{\A_\star}(H_\star, H_\star)$ is zero for $s > t - s + 1 + n$ and $s > \frac{1}{2}(m + 3)$.
\end{proof}

\begin{corollary}
\label{cor:tors-odd}
Let $\ell$ be an odd prime and $F$ a perfect field of characteristic not $\ell$.
If $t > 0$ and $t - w > 0$ the exponent of $\pi_{t,w}(\S_{\ell,\eta}^{\wedge})$ is bounded by $\ell^{\lceil(t - w)/(\ell - 2)\rceil}$.
\end{corollary}
\begin{proof}
By \Cref{cor:general-bockstein-ss} $\Ext^{s,(t, w)}_{\A_\star}(H_\star, H_\star)$ is a subquotient of
\[
\left\{\bigoplus_{(s,t,w) = (s',t',w') + (0, t'', w'')} \left(\frac{B^nH_\star}{B^{n+1}H_\star}\right)_{t'',w''} \tensor \Ext^{s',(t',w')}_{\Top} \right\}_n.
\]
As in \Cref{thm:ell-iso} we would need $2(t-w) \geq t'$ for this to be nonzero.
But if $(2\ell - 3)s > 2(t - w) - s \geq t' - s$ this is zero by \eqref{eq:adams-iso}.
That is, if $(\ell - 2)s > t - w$.
Hence the exponent of $\pi_{t,w}(\S_{\ell,\eta}^{\wedge})$ is bounded by $\ell^{\lceil(t - w)/(\ell - 2)\rceil}$.
\end{proof}

\section{The zeroth motivic stable stem}
\label{sec:zero-line}
In this section we compute the $(2,\eta)$-completed motivic zero line over perfect fields of characteristic not 2.
That is, $\lim_k K^{MW}_{-n}/(2^k, I^k) \to \oplus_n \pi_{n,n}(\S_{2,\eta}^{\wedge})$ is an isomorphism.
The computation is essentially the same as Morel's computation of $\pi_{0,0}(\S)$ in \cite{Morel:adams}. %
We think it is interesting that Morel's pull-back square \cite[Theoreme 5.3]{Morel:witt} for Milnor-Witt $K$-theory shows up naturally as part of the computation. 
That is, from \eqref{eq:MW=0} we get that the $E_\infty$-page converging to $\pi_{-n,-n}(\S_{2,\eta}^{\wedge})$ is a direct sum of the filtration quotients of Milnor K-theory with the filtration quotients of the fundamental ideal filtration,
glued together at the first stage of the filtration.
We write $E_{2}^{MW=m}$ for $\oplus_{t-s-w=m}E_2^{s,t,w}$.

Recall that the $n$'th power of the fundamental ideal is generated by the $n$-fold Pfister forms.
There is a canonical ring map \cite[p.~253]{Morel:pi0}
\[
\Phi: K^{MW}_{-*} \to \pi_{*,*}(\S),
\]
which maps a compatible pair of generators $(\!\langle\!\langle a_1,\dots,a_{n+k} \rangle\!\rangle, [a_1',\dots,a_n']) \in I^{n}\times_{I^{n}/I^{n+1}} K^M_n = K^{MW}_{n}$ to
\begin{equation}
\label{eq:MW-map}
\Phi(\!\langle\!\langle a_1,\dots,a_{n+k} \rangle\!\rangle, [a_1',\dots,a_n'])
=
\begin{cases}
[a_1',\dots,a_n'] & k = 0 \\
\eta^k[a_1,\dots a_{n+k}] & k > 0 
\end{cases}
\in \pi_{-n,-n}(\S),
\end{equation}
where $[a_1,\dots,a_n] = [a_1]\wedge\dots\wedge[a_n]$, for $[a] : S^{0,0}\to\Gm$ the map which sends the non-basepoint to $a$,
and $\eta\in \pi_{1,1}(\S)$ the Hopf-map.
We determine the filtration of $\pi_{n,n}(\S_{H}^\wedge)$ explicitly,
and show that $\Phi$ induces an isomorphism after $(2,I)$-completion.
Note that to identify $\pi_{n,n}(\S_{H}^\wedge)$ with $\pi_{n,n}(\S_{2,\eta}^\wedge)$ we need to know that $\pi_{*=*}(\S)/(2, \eta) \cong \pi_{*=*}(\HZ/2)$ \cite[Assumption 5.1]{Mantovani}, which as far as we know relies  on Morel's computation of $\pi_{*=*}(\S)$.

\begin{theorem}
\label{thm:zero-stem}
Over a perfect field of %
characteristic not 2, the canonical map $\Phi: K^{MW}_{-n} \to
\pi_{n,n}(\S)$ induces an isomorphism $\hat{\Phi} : \lim_s
K^{MW}_n/(2^s, I^s) \xrightarrow{\cong}
\pi_{-n,-n}(\S_{2,\eta}^\wedge)$.
\end{theorem}
\begin{proof}
In Milnor-Witt degree 0 \Cref{lem:ext} tells us that the $E_2$-page of the motivic Adams spectral sequence for the sphere spectrum converging to $\oplus_n \pi_{n,n}(\S^{\wedge}_{2,\eta})$ is
\begin{equation}
\label{eq:MW=0}
E_2^{MW=0} = k_*[h_1, h_0]/(\rho h_0, h_0h_1).
\end{equation}
Since the differentials are $h_1$-linear there cannot be any entering differentials from $E_2^{MW=1}$ which hits the subalgebra
$
k_*[h_1] \subset E_2^{MW=0}.
$
Furthermore, for degree reasons there cannot be any differentials on the
$\Ext^{s}_{\A_\star^\R}(H_\star^\R, H_\star^\R) \tensor_{\Z/2[\rho]} k_*$ part of $E_2^{MW=1}$, cf.~\cite[Figure 3]{DI:Real}.
Hence, there can only be differentials from ${}_{\rho}k_*[h_0] \subset E_2^{MW=1}$ to $k_*[h_0]/\rho h_0 \subset E_2^{MW=0}$.
But on these $k_*$-submodules we have an isomorphism with the $E_2$-page of the motivic Adams spectral sequence for $\HZ$.
That is, the bottom horizontal arrow in the diagram below is an isomorphism
\[
\begin{tikzcd}
E_2^{MW=1}(\S) \ar[r] &  E_2^{MW=1}(\HZ) \\
{}_{\rho}k_*[h_0] \ar[r, "\cong"]\ar[u, hook] & {}_{\rho}k_*[h_0], \ar[u, "="]
\end{tikzcd}
\]
and similarly for $E_2^{MW=0}$ and $k_*[h_0]/\rho h_0$.
Hence, the differentials are the same on ${}_{\rho}k_*[h_0]$,
and we get the following description of the $E_\infty^{MW=0}$-page
\[
E_\infty^{s,t,w}(\S) \cong k_{2s-t} \oplus E_{\infty}^{s,t,w}(\HZ)
\cong k_{2s-t} \oplus 2^s K^M_{s-t}/2^{s+1}K^M_{s-t}, s > 0,\, t - s- w = 0, w=-n,
\]
and $E_\infty^{0,-n,-n}(\S) \cong k_{n}$.
The identification of $E_{\infty}^{s,t,w}(\HZ)$ was obtained in \Cref{sec:Bockstein}.

Consider the filtered group
\[
F^0 = K^{MW}_n = I^{n}\times_{I^n/I^{n+1}}K^M_n,\qquad
F^s = I^{n+s} \oplus 2^s K^M_n,\, s > 0,
\]
with filtration quotients
\[
F^{s}/F^{s+1} = I^{n+s}/I^{n+s+1}\times_{I^n/I^{n+1}}2^s K^M_n/2^{s+1}K^M_n = I^{n+s}/I^{n+s+1}\oplus 2^s K^M_n/2^{s+1}K^M_n,\, s>0.
\]
We want to show that $\Phi$ induces a map of filtered groups with isomorphic filtration quotients.

Let $F^s\S$ be the Adams filtration of $\pi_{-n,-n}(\S)$.
By definition $\Phi(F^0) \subset F^1\S$.
Assume inductively $\Phi(F^s) \subset F^{s}\S$. We want to show $\Phi(F^{s+1}) \subset F^{s+1}\S$.
It suffices to show that $\Phi(F^{s+1})$ maps to zero in $F^s\S/F^{s+1}\S \cong E_\infty^{s,s-n,-n}$.
That is,
the composite $f: I^{n+s+1}\oplus 2^{s+1}K^M_n \to F^s\S \to F^s\S/F^{s+1}\S \cong E_\infty^{s,s-n,-n}(\S) \cong k_{n+s} \oplus 2^{s}K^M_n/2^{s+1}K^M_n$ is zero.
Since $E_\infty^{s,s-n,-n}(\S) \to E_\infty^{s,s-n,-n}(\HZ)$ is the identity map on $2^{s}K^M_n/2^{s+1}K^M_n$,
the projection of $f$ to $2^{s}K^M_n/2^{s+1}K^M_n$ is zero on $2^{s+1}K^M_n$.
The projections is also zero on $x \in I^{n+s+1}$, since $x$ is mapped to a multiple of $\eta$ by \eqref{eq:MW-map}, and $\eta$ is zero on $\HZ$.
Hence it suffices to show that the projection of $f$ to $k_{n+s}$ is zero.
For a generator $\langle\!\langle a_1,\dots,a_{n+s+1}\rangle\!\rangle \in I^{n+s+1}$ this follows from \Cref{lem:kn+s}.
A generator $2^{s+1}x \in 2^{s+1}K^M_n$ is mapped by $\Phi$ to an element of $\pi_{-n,-n}(\S)$ in Adams filtration $s+1$, hence $f(2^{s+1}x)$ is zero in $k_{n+s}$.
So $\Phi$ is a map of filtered groups.

Furthermore $\Phi$ induces an isomorphism on the filtration quotients. %
Indeed, $I^{n+s}/I^{n+s+1} \cong k_{n+s}$ by \Cref{lem:kn+s}.
The composite $K^{MW}_n \to \pi_{-n,-n}(\S) \to \pi_{-n,-n}(\HZ) = K^{M}_n$ is the canonical projection, so by \Cref{sec:Bockstein} the filtration quotients are isomorphic on $2^s K^M_n/2^{s+1}K^M_n$ as well.

Finally, if we complete $K^{MW}_n$ with respect to its filtration (if the field has finite virtual cohomological dimension this is 2-completion) $\Phi$ induces an isomorphism \cite[Theorem 2.6]{Boardman}.
That is, we have an isomorphism
\[
\hat{\Phi} : \lim_s K^{MW}_n/(2^s, I^s) \xrightarrow{\cong} \pi_{-n,-n}(\S_{2,\eta}^\wedge),
\]
since \eqref{eq:MW=0} converges strongly by \Cref{lem:strong-MW0}.
\end{proof}
\begin{lemma}
\label{lem:kn+s}
The composite $I^{n+s} \to F^s\S \to k_{n+s} \cong I^{n+s}/I^{n+s+1}$,
is the canonical map induced by sending a Pfister form $\langle\!\langle a\rangle\!\rangle$ to $[a]$.
\end{lemma}
\begin{proof}
For $n + s \leq 0$ this is immediate.
By \eqref{eq:MW-map} a Pfister form $\langle\!\langle a_1,\dots,a_{n+s}\rangle\!\rangle \in I^{n+s}$
is mapped to $\eta^{s}[a_1, \dots, a_{n+s}] \in \pi_{-n,-n}(\S)$.
The element $\eta^s$ has a lift to Adams-filtration $s$, which is given by some map $[h_1]^s : S^{s,s} \to \Sigma^{-s,0}\ol{H}^{\wedge s}$.
Taking the composite with $\Sigma^{-s,0}\ol{H}^{\wedge s} \to \Sigma^{-s,0}\ol{H}^{\wedge s}\wedge H$ we get an element representing $h_1^s$ on the $E_1$-page.
Smashing with $[a_1, \dots, a_{n+s}]$ then shows that $\eta^s [a_1, \dots, a_{n+s}]$ is mapped to $h_1^s[a_1, \dots a_{n+s}]$ on the $E_1$-page. I.e., we have a diagram
\[
\begin{tikzcd}
S^{n,n} \ar[r, "="]\ar[d, dashed, "{[h_1]^s[\ul{a}]}"]\ar[dd, bend right=50, "{h_1^s[\ul{a}]}" left] & S^{n,n}\ar[d, "{\eta^s[\ul{a}]}"] \\
\Sigma^{-s,0}\ol{H}^{\wedge s} \ar[r]\ar[d] & S \\
\Sigma^{-s,0}\ol{H}^{\wedge s}\wedge H,
\end{tikzcd}
\]
where the vertical composite is $h_1^s[a_1, \dots, a_n]$.
This element is a permanent cycle, so $\eta^s [a_1, \dots, a_{n+s}]$ is mapped to $h_1^s[a_1, \dots a_{n+1}]$ on the $E_\infty$-page. That is, $\langle\!\langle a_1,\dots,a_{n+s}\rangle\!\rangle$ is mapped to $[a_1, \dots, a_{n+s}] \in k_{n+s}$.
\end{proof}

\begin{remark}
At an odd prime $\ell$ the corresponding computation would give $\lim_k K^M_{n}/\ell^k \cong \pi_{-n,-n}(\S_{\ell,\eta}^{\wedge})$.
This is as expected, since the completion of $W(F)$ at the fundamental ideal and $\ell$ is zero.
For instance for $n < 0$, $\pi_{-n,-n}(\S) = W(F)$ by Morel's theorem \cite[Theorem 6.2.1]{Morel:pi0},
and
$\pi_{-n,-n}(\S_{\ell,\eta}^\wedge)
= \pi_{-n,-n}(\S)_{\ell,\eta}^{\wedge}
= \lim_k W(F)/(\ell^k, I^k)
$.
However $(\ell^k, I^k) = W(F)$ for all $k$, since $2^k \in I^k$. So $\pi_{-n,-n}(\S_{\ell,\eta}^\wedge) = 0$.
We expect the $\wHZ/\ell$-based motivic Adams spectral sequence to fix this, since this spectral sequence is convergent to the $\ell$-completion by the work of Mantovani \cite{Mantovani}.
This spectral sequence should be essentially identical to the $\HZ/\ell$-based motivic Adams spectral sequence,
but with a small modification in Milnor-Witt degree 0 to take care of the Witt-groups.
\end{remark}

\printbibliography

\end{document}